\documentclass[12pt]{amsart}
\usepackage{pifont}
\usepackage[bbgreekl]{mathbbol}
\usepackage{amsfonts}
\usepackage{mathrsfs}

\DeclareSymbolFontAlphabet{\mathbb}{AMSb}
\DeclareSymbolFontAlphabet{\mathbbl}{bbold}
\usepackage{dsfont}
\usepackage{anysize}
\marginsize{2.7cm}{2.44cm}{2.0cm}{3.0cm}
\usepackage[small,it]{caption}
\usepackage[OT2,T1]{fontenc}
\DeclareSymbolFont{cyrletters}{OT2}{wncyr}{m}{n}
\DeclareMathSymbol{\Sha}{\mathalpha}{cyrletters}{"58}
 \usepackage[latin1]{inputenc}
 \usepackage[dvips]{graphicx}
 \usepackage{wrapfig}
 \usepackage{amsmath}
 \usepackage{amsthm}
 \usepackage{amsfonts}
 \usepackage{amssymb}
 \usepackage{layout} \usepackage{verbatim}
 \usepackage{alltt}
\usepackage{yfonts}
\usepackage{color}

\usepackage[all]{xy}
\usepackage{setspace}

\newcounter{defcounter}
\newtheorem*{claim}{Claim}

\newcommand{\LL}{\Lambda}
\newcommand{\TT}{\mathbb{T}}
\newcommand{\QQ}{\mathbb{Q}}
\newcommand{\FF}{\mathcal{F}}

\newcommand{\lra}{\longrightarrow}
\newcommand{\ZZ}{\mathbb{Z}}
\newcommand{\PP}{\mathcal{P}}

\newcommand{\Gal}{\textup{Gal}}
\newcommand{\KS}{\textbf{\textup{KS}}}

\newcommand{\ra}{\rightarrow}

\newcommand{\be}{\begin{equation}}
\newcommand{\ee}{\end{equation}}

\newcommand{\al}{\mathcal{L}}

\newcommand{\oo}{\mathcal{O}}

\newcommand{\PPP}{\frak{P}}

\newcommand{\mm}{\hbox{\frakfamily m}}

\newcommand{\FFc}{\mathcal{F}_{\textup{\lowercase{can}}}}

\newcommand{\BK}{\mathbf{BK}}
\newcommand{\BBK}{\mathbb{BK}}
\newcommand{\Gr}{\textup{Gr}}

\newcommand{\crs}{\frak{C}}

\newcommand{\res}{\mathrm{res}}

\newcommand{\Iw}{\mathrm{Iw}}

\newcommand{\cyc}{\mathrm{cyc}}

\numberwithin{equation}{section}
\newtheorem{thm}{Theorem}[section]

\newenvironment{define}{\par\medskip\noindent\refstepcounter{thm}
\bgroup{\hspace*{-0.15 cm}\bf{Definition}
\thethm.}\bgroup}{\egroup \egroup\par\medskip}
\newtheorem{prop}[thm]{Proposition}
\newtheorem{cor}[thm]{Corollary}
\newenvironment{rem}{\par\medskip\noindent\refstepcounter{thm}
\bgroup{\hspace*{-0.15 cm}\bf{Remark} \thethm.}\bgroup}{\egroup
\egroup\par\medskip} \parskip 2pt

\newenvironment{conj}{\par\medskip\noindent\refstepcounter{thm}
\bgroup{\hspace*{-0.15 cm}\bf{Conjecture}
\thethm.}\bgroup}{\egroup \egroup\par\medskip}
\newenvironment{example}{\par\medskip\noindent\refstepcounter{thm}
\bgroup{\hspace*{-0.15 cm}\bf{Example}
\thethm.}\bgroup}{\egroup \egroup\par\medskip}

\newcounter{Athm}[section]

\renewcommand{\theAthm} {\arabic{Athm}}

\long\def\symbolfootnote[#1]#2{\begingroup%
\def\thefootnote{\fnsymbol{footnote}}\footnote[#1]{#2}\endgroup}

\begin{document}

\title[Perrin-Riou's conjecture]{B\lowercase{eilinson}-K\lowercase{ato} \lowercase{and} B\lowercase{eilinson}-F\lowercase{lach elements}, C\lowercase{oleman}-R\lowercase{ubin}-S\lowercase{tark classes}, H\lowercase{eegner} \lowercase{points} \lowercase{and} \lowercase{the} {P}\lowercase{errin}-{R}\lowercase{iou} {C}\lowercase{onjecture} }

\author{K\^az\i m B\"uy\"ukboduk}

\address{K\^az\i m B\"uy\"ukboduk\newline UCD School of Mathematics and Statistics\\ University College Dublin\\ Ireland
\newline Ko\c{c} University, Mathematics  \\
Rumeli Feneri Yolu, 34450 \\ 
Istanbul, Turkey}
\email{kazim@math.harvard.edu}

\keywords{Abelian Varieties, Modular Forms, Iwasawa Theory, Birch and Swinnerton-Dyer Conjecture}

\maketitle
\begin{abstract}
Our first goal in this note is to explain that a weak form of Perrin-Riou's conjecture on the non-triviality of Beilinson-Kato classes follows as an easy consequence of the Iwasawa main conjectures, and deduce its refined versions in the supersingular case from this fact and a variety of Gross-Zagier formulae. 

Our second goal is to set up a conceptual framework in the context of $\LL$-adic Kolyvagin systems to treat analogues of Perrin-Riou's conjectures for higher  motives of higher rank. We apply this general discussion in order to establish a link between Heegner points on a general class of CM abelian varieties and the (conjectural) Coleman-Rubin-Stark elements we introduce here. This is a higher dimensional version of Rubin's results on rational points on CM elliptic curves.
\end{abstract}
\section{Summary of Contents and Background}
\label{sec:intro}
We have three goals in this article. The first is to record a rather simple (but somehow overlooked) proof of Perrin-Riou's conjecture  (under very mild hypotheses) on the non-vanishing of the $p$-adic Beilinson-Kato class associated to an elliptic curve $E_{/\QQ}$, when $E$ has ordinary (i.e., good ordinary or multiplicative) or supersingular reduction at $p$. This generalizes (a fragment of) the forthcoming work of Bertolini and Darmon for a good ordinary prime $p$ and of Venerucci for a \emph{split} multiplicative prime $p$. In the supersingular case, we also explain how to deduce an explicit formula for a point of infinite order on $E(\QQ)$ in terms of the special values of the two $p$-adic $L$-functions (attached to two $p$-stabilizations of the associated eigenform). We remark that a similar formula was readily proved by Kurihara and Pollack assuming the truth of the $p$-adic Birch and Swinnerton-Dyer conjecture; our argument here is based on Kobayashi's $p$-adic Gross-Zagier formula (which essentially settles the $p$-adic Birch and Swinnerton-Dyer conjecture in this set up).

As a side benefit (which was our initial motivation to release the first portion of this note), we remark that these results (in the case of multiplicative reduction) render our previous work~\cite{kbbmtt} on the Mazur-Tate-Teitelbaum conjecture fully self-contained and in some sense, they simplify the proof of the main results therein: The results here allow us to by-pass the need to appeal to a two-variable exceptional zero formulea, as considered in \cite{venerucciinventiones}.

In the second part of this article (Section~\ref{subsec:KS}), we first recast this approach relying on the theory of $\LL$-adic Kolyvagin systems. We explain how this yields a proof of an extension of Perrin-Riou conjecture concerning the non-vanishing of the $p$-distinguished twists of Beilinson-Flach elements (Corollary~\ref{cor:PRBFelement}). 

In the third and final portion of this note, we establish a precise link between Heegner points on a general class of CM abelian varieties and the (conjectural) Coleman-Rubin-Stark elements we introduce here associated to these CM abelian varieties (c.f. Theorem \ref{thm:maincolemanRSheegnerPR} in the main text). This is a higher dimensional version of Rubin's results on rational points on CM elliptic curves (where he compares elliptic units to Heegner points on CM elliptic curves).

\subsubsection*{\textbf{{Part I. Perrin-Riou's conjecture for Beilinson-Kato elements.}}}Let $E$ be an elliptic curve defined over $\QQ$ and let $N$ denote its conductor. Fix a prime $p>3$ and let $S$ denote the set consisting of all rational primes dividing $Np$ and the archimedean place. In this set up, Kato \cite{ka1} has constructed an Euler system $\textbf{c}^{\textup{BK}}=\{c_F^{\textup{BK}}\}$ where $F$ runs through abelian extensions of $\QQ$, $c_F^{\textup{BK}}\in H^1(F,T_p(E))$ is unramified away from the primes dividing $Np$ and $T_p(E)$ is the $p$-adic Tate module of $E$. Kato's explicit reciprocity laws show that the class $c_\QQ^{\textup{BK}}\in H^1(\QQ,T_p(E))$ is non-crystalline at $p$ (and in particular, non-zero) precisely when $L(E/\QQ,1)\neq 0$, where $L(E/\QQ,s)$ is the Hasse-Weil $L$-function of $E$. Perrin-Riou in \cite[\S 3.3.2]{pr93grenoble} predicts the following assertion to hold true. Let $\textup{res}_p: H^1(G_{\QQ,S},T)\ra H^1(\QQ_p,T)$ denote the restriction map.
\begin{conj}
\label{conj:PR}
The class $\textup{res}_p\left(c_\QQ^{\textup{BK}}\right) \in H^1(\QQ_p,T_p(E))$ is non-torsion if and only if $L(E/\QQ,s)$ has at most a  simple zero at $s=1$. 
\end{conj}
This is the conjecture (and its extensions in other settings) we address in the current article. We will explain below how to deduce Conjecture~\ref{conj:PR} as an easy corollary of the work of Kato, Skinner-Urban and Wan on the main conjectures of Iwasawa theory of elliptic curves. 
\begin{thm}[Kato, Skinner-Urban, Wan]
\label{thm:main}
Suppose that $E$ is an elliptic curve such that the residual representation 
$$\overline{\rho}_E:\,G_{\QQ,S}\lra \textup{Aut}(E[p])$$
is surjective. Then the ``if" part of  Perrin-Riou's Conjecture~\ref{conj:PR} holds true in the following cases:
\begin{itemize}
\item[(a)] $E$ has good ordinary reduction at $p$.
\item[(b)] $E$ has good supersingular reduction at $p$ and $N$ is square-free.
\item[(c)] $E$ has multiplicative reduction at $p$ and there exists a prime $\ell \mid\mid N$ such that there $\overline{\rho}_E$ is ramified at $\ell$.
\end{itemize}
\end{thm}
As per the ``only if'' direction, one may deduce the following as a rather straightforward consequence of the recent results due to Skinner, Skinner-Zhang and Venerucci in the case of $p$-ordinary reduction and due to Kobayashi and Wan in the case of $p$-supersingular reduction. We state it here for the sake of completeness.
\begin{thm}[Skinner, Skinner-Zhang, Venerucci]
\label{thm:mainonlyif}
In the situation of Theorem~\ref{thm:main}, the ``only if'' part of Perrin-Riou's conjecture holds true for all cases $(\textup{a})$, $(\textup{b})$ and $(\textup{c})$ if we further assume: 
\begin{itemize}
\item in the case of $(\textup{a})$, that $N$ is square free and either $E$ has non-split multiplicative reduction at one odd prime or split multiplicative reduction at two odd primes; 
\item in the case of $(\textup{c})$ that
\begin{itemize}
\item $p$ does not divide $\textup{ord}_p(\Delta_E)$ and when $E$ split-multiplicative reduction at $p$, the Galois representation $E[p]$ is not finite at $p$,
 \item for all primes $\ell \mid\mid N$ such that $\ell \equiv \pm 1 \mod p$, the prime $p$ does not divide $\textup{ord}_\ell(\Delta_E)$,
\item there exists at least two prime factors $\ell \mid\mid N$ such that $p$ does not divide $\textup{ord}_\ell(\Delta_E)$.
 \end{itemize}
\end{itemize}
\end{thm}
We remark that in the situation of (a), the hypotheses in Theorem~\ref{thm:mainonlyif}  may be slightly altered if we relied on the work of Zhang~\cite[Theorem~{1.3}]{zhangconversetokolyCambridge} on the converse of the Gross-Zagier-Kolyvagin theorem, in place of the work of Skinner. This on one hand would allow us to relax the condition on the conductor $N$, on the other hand would force us to introduce additional hypothesis (see Theorem 1.1 of loc.cit). 

In a variety of cases, we will be able refine Theorem~\ref{thm:main} and deduce that the square of the logarithm of a suitable Heegner point agrees with the logarithm of the Beilinson-Kato class $\BK_1$ up to an explicit non-zero algebraic factor and verify some of the hypothetical conclusions in \cite[\S3.3.3]{pr93grenoble}\footnote{Given Theorem~\ref{thm:main} and Kobayashi's $p$-adic Gross-Zagier formula, Perrin-Riou's article seems to contain an unconditional proof of Corollaries~\ref{cor:introsupersingularPRenhanced} and \ref{cor:introRubinsformula}. After an e-mail exchange with F. Castella, we came to believe that it would be useful to provide a somewhat detailed discussion on these results (which benefits from the recent developments in the theory of triangulordinary Selmer groups). We are grateful to F. Castella for his inquiry regarding this point.}. We record here the following three results in the cases when $E$ has good supersingular or bad non-split multiplicative reduction at $p$ (see Remark~\ref{rem:goodordcase} below for the developments concerning the good ordinary case). For a detailed discussion and proofs, we refer the reader to Section~\ref{sec:heegner}.
\begin{cor}[Perrin-Riou]
\label{cor:introsupersingularPRenhanced}
Suppose that $E$ has good supersingular reduction at $p$  and verifies the hypotheses of Theorem~\ref{thm:main}, as well as that its conductor $N$ is square-free. Then,
$$\log_E\left(\textup{res}_p(\BK_1)\right)=-(1-1/\alpha)(1-1/\beta)\cdot C(E)\cdot\log_E\left(\textup{res}_p(P)\right)^2$$
for a suitably chosen Heegner point $P \in E(\QQ)$, where $\log_E$ stands for the coordinate of the Bloch-Kato logarithm associated to $E$ with respect to a suitably normalized N\'eron differential on $E$ and $C(E)\in \QQ^\times$ is given in (\ref{eqn:GZ}).
\end{cor}
This is Theorem~\ref{thm:mainsupersingularexplicit} below. Its proof that we present here follows the general view point that the works \cite{benoisheights,  benoiscomplex, benoisbuyukboduk} offer on the theory of heights on triangulordinary Selmer groups, which in a certain (perhaps very subjective) sense of the word simplifies the discussion in  \cite[\S4]{prGZ} and \cite[\S3.3.3]{pr93grenoble}. 
\begin{rem}
\label{rem:goodordcase}
The treatment of the good ordinary case with the strategy outlined here requires a $p$-adic Gross-Zagier formula at critical slope. After recent developments (that we shall outline below), this formulae is now within our reach and it is the subject of our forthcoming joint work with R. Pollack and S. Sasaki~\cite{BPS-CriticalGZ}. 

We sketch here our strategy in \cite{BPS-CriticalGZ} to prove the desired Gross-Zagier formula for the critical slope $p$-adic $L$-function attached to the $p$-stabilization $f^{\beta}_E$, where $f_E$ is the eigenform attached to $E$ and $v_p(\beta)=1$. As the first step one considers a Coleman family ${\mathbf f}$ deforming $f$ over a suitable affinoid $A$, with constant slope $1$ and $U_p$-eigenvalue ${\bbbeta}$. Let $\mathscr{Z}$ denote the set of integers $k>2$ which are congruent to $2$ modulo $p-1$. For $k\in \mathscr{Z}$, we let ${\bf{f}}(k)$ denote the specialization of $\bf{f}$ with trivial wild character, so that ${\bf{f}}(k)$ is a $p$-stabilized $p$-old eigenform, which is the $p$-stabilization of a newform ${\bf{f}}(k)^{\circ}$ of level $N$ associated to the root $\bbbeta(k)$ of the Hecke polynomial of ${\bf{f}}(k)^{\circ}$ at $p$. Notice that ${\bf{f}}(k)$ is of non-critical slope. Moreover, $\bbbeta(k)$ has smaller $p$-adic valuation than the other root of the Hecke polynomial of ${\bf{f}}(k)^{\circ}$ at $p$ whenever $k\geq 4$.

The first step is to interpolate Heegner cycles (or rather the Generalized Heegner cycles of Bertolini-Darmon-Prasanna) in our Coleman family; this is one of the tasks we will be carrying out in our forthcoming article. We note that Jetchev-Loeffler-Zerbes have independently announced the construction of big Generalized Heegner cycles over a Coleman family (seemingly with a method different than ours). 

The second step is to prove a $p$-adic Gross-Zagier formulae for non-ordinary newforms ${\bf{f}}(k)^{\circ}$, corresponding to their stabilization with respect to the eigenvalue $\bbbeta(k)$. This has been recently announced by S. Kobayashi.

Thanks to the work of Benois, $p$-adic (cyclotomic) height pairings that appear in step two interpolate along the Coleman family $\bf{f}$, giving rise to an $A$-adic (cyclotomic) height pairings. Combining this fact with Steps 1 and 2 (along with the density of $\mathscr{Z}$ in $A$), one obtains an $A$-adic Gross-Zagier formula for the two variable $p$-adic $L$-function $L_p({\bf f}/K):=L_p({\bf f})L_p({\bf f}\otimes \chi_K)$ (where the two $p$-adic $L$-functions on the right are those constructed Pollack-Stevens and Bella\"{i}che) expressing its cyclotomic derivative of $L_p({\bf f}/K)$ as the $A$-adic height of the big Generalized Heegner cycle. On specializing to $f^\beta_E={\bf f}(2)$, we obtain the desired $p$-adic Gross-Zagier formula at critical slope.

Before we end this remark, we note that Bertolini and Darmon have readily announced a proof of Corollary~\ref{cor:introsupersingularPRenhanced} in the case of good ordinary reduction. Their methods are disjoint from what we have sketched within this remark (though see also Remark~\ref{rem:darmonvenerucci} below).
\end{rem}

Corollary~\ref{cor:introsupersingularPRenhanced} combined with the Gross-Zagier formula, Kobayashi's $p$-adic Gross-Zagier formula in this set up and Perrin-Riou's analysis in \cite[\S2.2.2]{pr93grenoble} yields the following result, which allows us to determine a global point in terms of the special values of the associated $p$-adic $L$-functions and validates Formula 3.3.4 of \cite{pr93grenoble}.  Let $\omega_\alpha,\omega_\beta \in D_{\textup{cris}}(V)$ denote the canonical elements given as in Section~\ref{subsubsec:Ehasgoodreductionatp}. We set $\delta_E:=[\omega_\beta,\omega_\alpha]/C(E)$, where $[\,,\,]: D_{\textup{cris}}(V)\times D_{\textup{cris}}(V)\ra \QQ_p$ is the canonical pairing. 
\begin{cor}[Gross-Zagier, Kobayashi, Perrin-Riou]
\label{cor:introRubinsformula}
Let $\exp_V$ denote the Bloch-Kato exponential map. Then under the hypotheses of Corollary~\ref{cor:introsupersingularPRenhanced},
$$P:=\exp_{V}\left(\omega^*\cdot\sqrt{\delta_E\left((1-1/\alpha)^{-2}\cdot L_{p,\alpha}^\prime(E/\QQ,1)-(1-1/\beta)^{-2}\cdot L_{p,\beta}^\prime(E/\QQ,1)\right)}\right)$$
is a $\QQ$-rational point on $E$ of infinite order.
\end{cor}
See \cite[Section 2.7]{kuriharapollack} (and of course, Perrin-Riou's article \cite{pr93grenoble}) for a similar formula which is verified assuming the truth of the $p$-adic Birch and Swinnerton-Dyer conjecture. We instead rely on a Rubin-style formula and the Gross-Zagier formulae of \cite{grosszagier,kobayashiGZ}. The reader will of course notice that Kobayashi's work in op. cit. implies the $p$-adic Birch and Swinnerton-Dyer conjecture up to non-zero rational constants.

We may also deduce the following version of Corollary~\ref{cor:introsupersingularPRenhanced} (relying on Disegni's work in place of Kobayashi's) in the case when $E$ has non-split multiplicative reduction at $p$.
\begin{cor}[Disegni, Gross-Zagier, Nekov\'a\v{r}]
\label{cor:mainenhanced}
Suppose that $E$ is an elliptic curve with non-split-multiplicative reduction at $p$ and verifies the hypotheses of Theorem~\ref{thm:main}, also that there exists a prime $\ell \mid\mid N$ such that there $\overline{\rho}_E$ is ramified at $\ell$. Assume that $r_{\textup{an}}=1$ and  further that Nekov\'a\v{r}'s $p$-adic height pairing associated to the canonical splitting of the Hodge-filtration on the semi-stable Dieudonn\'e module $D_{\textup{st}}(V)$ is non-vanishing. Then,
$$\log_E\left(\textup{res}_p(\BK_1)\right)\cdot\log_E\left(\textup{res}_p(P)\right)^{-2}\in \overline{\QQ}^\times\,,$$
where $P$ is any generator of $E(\QQ)/E(\QQ)_{\textup{tor}}$\,.
\end{cor}

\begin{rem}
\label{rem:darmonvenerucci}
Bertolini and Darmon have announced that in their forthcoming work \cite{BDBSD3}, they prove a result similar to Corollaries \ref{cor:introsupersingularPRenhanced} and \ref{cor:mainenhanced} in the situation of (a). Also, using different techniques then those of \cite{BDBSD3}, Venerucci \cite{venerucciinventiones} gave a proof of the result above when $E$ has split multiplicative reduction at $p$. They may then deduce the conclusions of Theorems~\ref{thm:main} and \ref{thm:mainonlyif} directly from these results (in the respective setting). In contrast, our exposition here is based on the observation that the weak Perrin-Riou Conjecture~\ref{conj:PR} on the non-vanishing of the Beilinson-Kato class is an immediate consequence of Iwasawa main conjectures. Following Perrin-Riou's original strategy, we then exploit a variety of Gross-Zagier formulae to deduce the full Perrin-Riou conjecture (that relates the Beilinson-Kato class to Heegner points). In particular, the proof of full Perrin-Riou conjecture in the situation of (a) also follows from Theorem~\ref{thm:main} thanks to a $p$-adic Gross-Zagier formula for the critical slope $p$-adic $L$-function (along with Perrin-Riou's $p$-adic Gross-Zagier formula for the slope-zero $p$-adic $L$-function), whose proof we have outlined in Remark~\ref{rem:goodordcase} and which is the subject of~\cite{BPS-CriticalGZ}.

Concerning Corollaries \ref{cor:introsupersingularPRenhanced} and \ref{cor:mainenhanced}, we would like to underline\footnote{We would like to thank Henri Darmon for an enlightening exchange regarding this point.} a common key feature of the three approaches (in \cite{BDBSD3, venerucciinventiones} and the original approach of Perrin-Riou that we take as a base here) towards it, despite their apparent differences:  All three works make crucial use of a suitable $p$-adic Gross-Zagier formula, allowing the comparison of Heegner points with Beilinson-Kato elements. For the approach in \cite{BDBSD3}, this formula is provided by \cite{BDprasannaduke} (where the relevant $p$-adic Gross-Zagier formula is proved by exploiting Waldspurger's formula and it resembles Katz's proof of the $p$-adic Kronecker limit formula) and for Venerucci's approach in \cite{venerucciinventiones}, it is provided by \cite{BDHida2007} (where the authors use Hida deformations and the \v{C}erednik-Drinfeld uniformization of Shimura curves). In the proof of Corollary~\ref{cor:mainenhanced} here, we rely on the $p$-adic Gross-Zagier formulae of Perrin-Riou~\cite{prGZ} in the situation of (a), of Kobayashi~\cite{kobayashiGZ} in the situation of (b) and the recent work of Disegni~\cite{disegniGZ} when $E$ has non-split-multiplicative  reduction at $p$.
\end{rem}   
\begin{rem}
\label{rem:selfdualmodforms}
Our arguments here easily adapt to treat also higher weight eigenforms; however, our conclusion in that situation is not as satisfactory as in the case of elliptic curves. For this reason, here we shall only provide a brief overview of our results towards Perrin-Riou's conjecture in that level of generality.  We say that a Galois representation $V$ (with coefficients in a finite extension $K$ of $\QQ_p$) is \emph{essentially self-dual} if it has a self-dual Tate-twist $V(r)$ and we say that an elliptic eigenform $f$ (of even weight $2k$ and level $N$, with $N$ coprime to $p$) is essentially self-dual if Deligne's representation $W_f$ associated to $f$ is. In this case, we set $V_f=W_f(k)$; this is necessarily the self-dual twist of $W_f$. Fix a Galois-stable $\frak{o}_K$-lattice $T_f$ contained in $V_f$ and let $\texttt{k}$ denote the residue field of $\frak{o}_K$. We will set $\rho_f: G_{\QQ,S}\ra \textup{GL}(T_f)$ (where $S$ is the set consisting of all rational primes dividing $Np$ and the archimedean place) and $\overline{\rho}_f:= \rho_f\otimes \texttt{k}$. If the conditions that
\begin{itemize}
\item $\overline{\rho}_f$ is absolutely irreducible,
\item $f$ is $p$-distinguished (namely, the semi-simplification of $\overline{\rho}_f\big{|}_{G_{\QQ_p}}$ is non-scalar),
\end{itemize}
simultaneously hold true, then 
\begin{align*}\textup{ord}_{s=k}\,L(f,s)=1 \implies& \hbox{ either } \log_{V_f}\BK_1\neq 0\,,\\
& \hbox{ or else \,} \textup{res}_p:H^1_f(\QQ,V_f)\ra H^1_f(\QQ_p,V_f) \hbox{ is the zero map}.\end{align*}
Here $H^1_f(\QQ_p,V_f) \subset H^1(\QQ_p,V_f)$ is the image of the Bloch-Kato exponential map $\exp_{V_f}$ and $H^1_f(\QQ,V_f)$ is the Bloch-Kato Selmer group. 

Note in particular that the main result of \cite{benoisbuyukboduk} towards Perrin-Riou's conjecture for $p$-non-crystalline semistable modular forms escapes the methods of the current article.
\end{rem}
\subsubsection*{\textbf{Part II. $\LL$-adic Kolyvagin systems, Beilinson-Flach elements and Coleman-Rubin-Stark elements}} The general theory of $\LL$-adic Kolyvagin systems yields a relatively simple criterion to verify the Perrin-Riou conjecture in a great level of generality.  This is the content of our Theorem~\ref{thm:PRwithKS} below. Combined with the work of Wan~\cite{wanpordinaryindefinite} and Howard~\cite{benGL2, bentwistedGZ}, it yields an easy proof of the following statement (which is Corollary~\ref{cor:PRBFelement} in the main body of this note):
\begin{thm}
\label{thm:introPRBF}
Let $E/\QQ$ be a non-CM elliptic curve that has good ordinary reduction at $p$ and assume that the residual representation $\overline{\rho}_E$ is absolutely irreducible. Let $K/\QQ$ be an imaginary quadratic extension that satisfies the weak Heegner hypothesis for $E$ and where $p$ splits completely. Suppose further that $p$ does not divide $\textup{ord}_\ell(j(E))$ whenever $\ell$ is a prime of split multiplicative reduction. If the twisted $L$-function $L(E/K,\alpha,s)$ associated to the base change of $E/K$ and a $p$-distinguished character $\alpha$ vanishes at $s=1$ to exact order $1$, then the corresponding Beilinson-Flach element $\textup{BF}_1 \in H^1_f(K,T_p(E)\otimes\alpha^{-1})$ is non-trivial.
\end{thm}
We remark that Bertolini and Darmon have announced the proof of an analogue of full Perrin-Riou conjecture in this set up, that enables them to compare the logarithms of Beilinson-Flach elements and Heegner points. Theorem~\ref{thm:introPRBF} is a considerably weaker version of their result. On the other hand, Theorem~\ref{thm:introPRBF} may be extended to cover the case of elliptic curves with good supersingular reduction as well, if one relied on another preprint of Wan on Iwasawa main conjectures at supersingular primes.

The view point offered by the theory of $\LL$-adic Kolyvagin systems enables us to address similar problems concerning a CM abelian variety $A$ of dimension $g$. In this situation, a slightly stronger version of Rubin-Stark conjectures equips us with a  rank-$g$ Euler system.  Making use of Perin-Riou's \emph{extended logarithm maps} and the methods of \cite{kbbCMabvar,kbbleiPLMS}, we may obtain Kolyvagin systems (Theorem~\ref{thm:mainRS}) out of these classes (that we call the \emph{Coleman-Rubin-Stark Kolyvagin systems}), with which we may apply Theorem~\ref{thm:PRwithKS}. Our \emph{Explicit Reciprocity Conjecture}~\ref{conj:explicitreciprocityforRS} for these classes is a natural generalization of the Coates-Wiles explicit reciprocity law for elliptic units, and predicts in a rather precise manner how the conjectural Rubin-Stark elements should be related to the Hecke $L$-values attached to the CM abelian variety $A$.

All this combined with the results of \cite{kbbCMabvar} allows us to prove Theorem~\ref{thm:introPRColemanPR} below, which establishes an explicit link between the Coleman-Rubin-Stark elements and Heegner points.  Let $A$ be an abelian variety which has complex multiplication by an order of a CM field $K$ whose index in the maximal order is prime to $p$ and defined over the maximal totally real subfield $K_+$ of $K$. We assume that $A$ has good ordinary reduction at every prime above $p$,  that the prime $p$ is unramified in $K_+/\QQ$ and that $A$ verifies the non-anomaly hypothesis (\ref{eqn:hna}) at $p$. One may than associate $A$ a Hilbert modular CM form $\phi$ on $\textup{Res}_{K_+/\QQ}\,\textup{GL}_2$ of weight $2$. See Section~\ref{subsec:CMabvarPRStark} for a detailed discussion of these objects.

We will assume that there exists a degree one prime of $K_+$ above $p$ (we believe that it should be possible to get around of this assumption with more work). Let $\langle\,,\,\rangle$ denote the $p$-adic height pairing introduced in Definition~\ref{def:heightparing} and \emph{assume that it is non-zero}. Denote by $\frak{C}$ the Coleman-Rubin-Stark element associated to the CM form $\phi$  (that is given as in Definition~\ref{def:ColemanRS}).
\begin{thm}
\label{thm:introPRColemanPR}
If the Hecke $L$-series of the associated to CM form $\phi$ vanishes to exact order $1$ at $s=1$, then the Coleman-Rubin-Stark class $\frak{C}$ is non-trivial and we have
$$\log_{\omega}\left(\frak{C}\right) \equiv \log_{\omega}\left(P_\phi\right)^2 \mod \overline{\mathbb{Q}}^\times\,,$$
where $P_\phi$ is a Heegner point on $A$ and $\log_\omega$ is a certain coordinate of the Bloch-Kato logarithm $($with respect to a suitably chosen N\'eron differential form$)$ given as in $($\ref{eqn:logomegadefine}$)$ below. 
\end{thm}
See Theorem~\ref{thm:maincolemanRSheegnerPR} below for a more precise version of this statement. We remark that Burungale and Disegni~\cite{burungaledisegni} recently proved the generic non-triviality of $p$-adic heights. Relying on this result, one may generically by-pass this hypothesis for the twisted variants of Theorem~\ref{thm:introPRColemanPR}. 
\begin{rem}
We were informed by D. Disegni that in a future version of \cite{burungaledisegni}, the authors will be proving a complementary version of the formula in Theorem~\ref{thm:introPRColemanPR}, which will express the right hand side in terms of the special values of a suitable $p$-adic $L$-function.
\end{rem}
\subsection{Notation} For a number field $K$, we define $K_S$ to be the maximal extension of $K$ unramified outside a finite set of places  $S$ of $K$ that contains all archimedean places as well as all those lying above $p$. Set $G_{K,S}:=\Gal(K_S/K)$. 

Let $\bbmu_{p^{\infty}}$ denote the $p$-power roots of unity. For a complete local noetherian $\ZZ_p$-algebra $R$ and an $R[[G_{\QQ,S}]]$-module $X$ which is free of finite rank over $R$, we define $X^*:=\textup{Hom}(X, \bbmu_{p^{\infty}})$ and refer to it as the Cartier dual of $X$. For any ideal $I$ of $R$, we denote by $X[I]$ the $R$-submodule of $X$ killed by all elements of $I$.

 Let $A/K$ be an abelian variety of dimension $g$ which has good reduction outside $S$. We let $T_p(A)$ denote the $p$-adic Tate-module of $A$; this is a free $\ZZ_p$-module of rank $2g$ which is endowed with a continuous $G_{K,S}$-action.  We define the \emph{cyclotomic deformation} $\TT_p(A)$ of $T_p(A)$ by setting $\TT_p(A):=T_p(A)\otimes\LL$ (where we let $G_{K,S}$ act diagonally), where $\LL:=\ZZ_p[[\Gamma]]$ (with $\Gamma=\Gal(K_\cyc/K)$ is the Galois group of the cyclotomic $\ZZ_p$-extension  $K_\cyc/k$) is the cyclotomic Iwasawa algebra. We  let $K_n/K$ denote the unique subextension of $K_
 \cyc/K$ of degree $p^n$ (and Galois group $\Gamma_n:=\Gamma/\Gamma^{p^n}\cong \ZZ/p^n\ZZ$). When $K=\QQ$, we write $\QQ_\infty$ in place of $\QQ_\cyc$.

In the first part of this article the number field $K$ will be $\QQ$ whereas in the second part, it will be either a more general totally real field or a CM field. In Part II, we will also work with a general continuous $G_K$-representation $T$ which unramified outside $S$ and which is free of finite rank over a finite flat extension $\frak{o}$ of $\ZZ_p$. We will denote by $L$ the field of fractions of $\frak{o}$ and by $\TT$ the $G_{K,S}$-representation $T\otimes \LL$. 

Fix a topological generator $\gamma$ of $\Gamma$. We will denote by $\textup{pr}_0$ the augmentation map $\LL\ra \ZZ_p$ (which induced from $\gamma\mapsto1$) and by slight abuse, also any map induced by it. 
\subsubsection{Selmer structures}
\label{sec:SelmerandKS}
 Given a general Galois representation $\TT$, we let $\FF_\LL$ denote the \emph{canonical Selmer structure} on $\TT$ defined by setting $H^1_{\FF_\LL}(K_\lambda, \TT)=H^1(K_\lambda,\TT)$ for every prime $\lambda$ of $K$.

In the notation of \cite[Definition 2.1.1]{mr02} we have $\Sigma(\FF)=S$ for each of the Selmer structures above.

\begin{define}
\label{define:dualSelmer}
 For every prime $\lambda$ of $K$, there is the perfect local Tate pairing
$$\langle\,,\,\rangle_{\lambda,\textup{Tate}}\,:H^1(K_\lambda,X) \times
H^1(K_\lambda,X^*) \ra H^2(K_\lambda,\bbmu_{p^{\infty}})\stackrel{\sim}{\lra}\QQ_p/\ZZ_p,$$
where. For a Selmer structure $\FF$ on $X$, define the
\emph{dual Selmer structure} $\FF^*$  on $X^*$ by setting $H^1_{\FF^*}(\QQ_\ell,X^*):=H^1_\FF(\QQ_\ell,X)^\perp$, the orthogonal complement of $H^1_\FF(\QQ_\ell,X)$ with respect to the local Tate pairing. 
\end{define}

Given a Selmer structure $\FF$ on $X$, we may propagate it to the subquotients of $X$ (via \cite[Example 1.1.2]{mr02}). We denote the propagation of $\FF$ to a subquotient still by $\FF$.
\tableofcontents
\section{Part I. Proof of Perrin-Riou's conjecture for elliptic curves over $\QQ$}
\label{sec:PRconjproof} 
Let $E/\QQ$ be an elliptic curve as above and recall the Beilinson-Kato Euler system $\mathbf{c}^\BK=\{c_F^\BK\}$. We write 
$$\BBK_1:=\{c_{\QQ_n}^{\textup{BK}}\}\in \varprojlim H^1(G_{\QQ_n,S},T_p(E)) =H^1(\QQ,\TT_p(E)),$$
where the last equality follows from \cite[Lemma 5.3.1(iii)]{mr02}. We also set $\BK_1:=c_\QQ^{\textup{BK}}$ so that $\textup{pr}_0\left(\BBK_1\right)=\BK_1$. It follows from the non-vanishing results of Rohrlich~\cite{rohrlich84cyclo} and Kato's explicit reciprocity laws for the Beilinson-Kato elements that $\BBK_1$ never vanishes.

Let us denote the  order of vanishing of the Hasse-Weil $L$-function $L(E/\QQ,s)$ by $r_{\textup{an}}$ and call it the \emph{analytic rank of $E$}.

\subsection{Preliminaries}
\label{subsec:prelim}
We first explain how the ``only if'' part of Perrin-Riou's conjecture (Theorem~\ref{thm:mainonlyif}) may be deduced as an easy corollary of the work of Skinner, Skinner-Zhang  and Venerucci. The argument we present here involves some of the reduction steps which we rely on for the proof of the ``if'' part of the conjecture and we pin these down in this portion of our article. 
\begin{proof}[Proof of Theorem~\ref{thm:mainonlyif}]
Suppose first that $\textup{res}_p^s(\BK_1)\neq 0$, where $\textup{res}_p^s$ is the singular projection given as the compositum of the arrows
$$H^1(\QQ,T)\ra H^1(\QQ_p,T)\twoheadrightarrow H^1(\QQ_p,T)/H^1_f(\QQ_p,T)=:H^1_s(\QQ_p,T)\,.$$
Kato's explicit reciprocity law shows that $r_{\textup{an}}=0$. We may therefore assume without loss of generality that $\BK_1$ is crystalline at $p$, namely that  $\BK_1\in H^1_f(\QQ,T)$.

Since $\BK_1 \neq 0$, it follows from \cite[Corollary 5.2.13(i)]{mr02} that $H^1_{\FFc^*}(\QQ,T^*)$ is finite. Recall that $\FF_{\textup{str}}$ denotes the Selmer structure on $T$ given by
\begin{itemize}
\item $H_{\FF_{\textup{str}}}(\QQ_\ell,T)=H_{\FF}(\QQ_\ell,T)$, if $\ell\neq p$,
\item $H_{\FF_{\textup{str}}}(\QQ_p,T)=0$.
\end{itemize}
We contend to verify that $H^1_{\FF_{\textup{str}}}(\QQ,T)=0$. Assume on the contrary that $H^1_{\FF_{\textup{str}}}(\QQ,T)$ is non-trivial. Since module $H^1(G_{\QQ,S},T)$ is torsion free under our running hypothesis on the image of $\overline{\rho}_E$, this amounts to  saying that $H^1_{\FF_{\textup{str}}}(\QQ,T)$ has positive rank. 

Recall further that the propagation of the Selmer structure $\FF_{\textup{str}}$ to the quotients $T/p^nT$ (in the sense of \cite{mr02}) is still denoted by $\FF_{\textup{str}}$. Recall that $T^*\cong E[p^\infty]$ and note for any positive integer $n$ that we may identify the quotient $T/p^nT$ with $E[p^n]$. By \cite[Lemma 3.7.1]{mr02}, we have an injection
$$H_{\FF_{\textup{str}}}(\QQ,T)/p^nH_{\FF_{\textup{str}}}(\QQ,T) \hookrightarrow H_{\FF_{\textup{str}}}(\QQ,T/p^nT)=H_{\FF_{\textup{str}}}(\QQ,E[p^n])$$
induced from the projection $T\ra T/p^nT$. This shows that
\be\label{eqn:lowerboundonstr}
\textup{length}_{\ZZ_p}\left(H_{\FF_{\textup{str}}}(\QQ,E[p^n])\right)\geq n.
\ee
As above, we let $\FFc=\FF_1$ denote the canonical Selmer structure on $T$, given by
\begin{itemize}
\item $H_{\FF_{\textup{can}}}(\QQ_\ell,T)=H_{\FF}(\QQ_\ell,T)$, if $\ell\neq p$,
\item $H_{\FF_{\textup{can}}}(\QQ_p,T)=H^1(\QQ_p,T)$.
\end{itemize}
 It follows from \cite[Lemma I.3.8(i)]{r00} (together with the discussion in \cite[\S6.2]{mr02}) that we have an inclusion
$$H_{\FF_{\textup{str}}}(\QQ_\ell,E[p^n])\subset H_{\FF_{\textup{can}}^*}(\QQ_\ell,E[p^n])$$
for every $\ell$. Here, $E[p^n]$ is identified with $T$ on the left and with $T^*[p^n]$ on the right and also viewed as a submodule of $T\otimes \QQ_p/\ZZ_p$). Furthermore, the index of $H_{\FF_{\textup{str}}}(\QQ_\ell,E[p^n])$ within $H_{\FF_{\textup{can}}^*}(\QQ_\ell,E[p^n])$ is bounded independently of $n$ (in fact, bounded by the order of $E(\QQ_p)[p^\infty]$). This in turn shows that together with (\ref{eqn:lowerboundonstr}) that
\be
\label{eqn:lowerboundoncan}
\textup{length}_{\ZZ_p}\left(H_{\FF_{\textup{can}}^*}(\QQ,E[p^n])\right)\geq n.
\ee
However, $H_{\FF_{\textup{can}}^*}(\QQ,E[p^\infty])$ is finite and therefore the length of
$$H_{\FF_{\textup{can}}^*}(\QQ,E[p^n])\cong H_{\FF_{\textup{can}}^*}(\QQ,E[p^\infty])[p^n]$$
(where the isomorphism is thanks to \cite[Lemma 3.5.3]{mr02}, which holds true here owing to our assumption on the image of $\overline{\rho}_E$) is bounded independently of $n$. This contradicts (\ref{eqn:lowerboundoncan}) and shows that $H^1_{\FF_{\textup{str}}}(\QQ,T)=0$. Thence, the map 
$$\textup{res}_p:\,H^1_f(\QQ,T)\lra H^1_f(\QQ_p,T)$$
is injective. The module $H^1_f(\QQ_p,T)$ is free of rank one and we conclude that $H^1_f(\QQ,T)$ has also rank one. When we are in the situation of (a) or (c), the proof now follows from the converse of the Kolyvagin-Gross-Zagier theorem proved in \cite{skinnerKGZ,skinnerzhang}. In the situation of (b), it follows from the works Kobayashi~\cite{kobayashiGZ} and Wan~\cite{xinwansupersingularmainconj} that a suitable Heegner point $P \in E(\QQ)$ is non-trivial. This in turn implies (relying on the classical Gross-Zagier formula) that $r_{\textup{an}}=1$, as desired.
\end{proof}
\begin{prop}
\label{prop:strisfinite}
If $r_{\textup{an}}\leq 1$, then $H^1_{\FFc^*}(\QQ,T^*)$ is finite.
\end{prop}
\begin{proof}
If $r_\textup{an}=0$, it follows that $\textup{res}_p^s(\BK)$ and in particular that $\BK_1\neq0$. The conclusion of the proposition follows from \cite[Corollary 5.2.13(i)]{mr02}. Suppose now that $r_\textup{an}=1$. In this case, 
\begin{align*}
H^1_{\FF_\textup{str}}(\QQ,T)&=\ker(H^1_f(\QQ,T)\lra H^1_f(\QQ_p,T))\\
&=\ker\left(E(\QQ)\,\widehat{\otimes}\, \ZZ_p \lra E(\QQ_p)\,\widehat{\otimes}\, \ZZ_p\right)=0
\end{align*}
where the second equality follows from the finiteness of the Tate-Shafarevich group \cite{koly90groth} and the final equality  by the Gross-Zagier theorem. As in the proof of Theorem~\ref{thm:mainonlyif}, we may use this conclusion to deduce that the order of $H_{\FF_{\textup{can}}^*}(\QQ,T^*)[p^n]\cong H_{\FF_{\textup{can}}^*}(\QQ,E[p^n])$ is bounded independently of $n$. The proof follows. 
\end{proof}
\subsection{Main conjectures and Perrin-Riou's conjecture}
We recall in this section Kato's formulation of the Iwasawa main conjecture for the elliptic curve $E$ and record results towards this conjecture. It follows from Kato's reciprocity laws and Rohrlich's \cite{rohrlich84cyclo} non-vanishing theorems that the class $\BBK_1$ is non-vanishing and the $\LL$-module $H^1(\QQ,\TT)$ is of rank one (as it was predicted by the weak Leopoldt conjecture for $E$). 

For two ideals $I, J \subset \LL$, we write $I\doteqdot J$ to mean that $I=p^{e} J$ for some integer $e$.
\begin{conj}
\label{ref:conjKatomain}
$\textup{char}\left(H^1_{\FF_\LL^*}(\QQ,\TT^*)^\vee\right)\doteqdot\textup{char}\left(H^1(\QQ,\TT)/\LL\cdot\BBK_1\right)$\,.
\end{conj}
This assertion is equivalent (via Kato's reciprocity laws, and up to powers of $p$) to the classical formulation of Iwasawa main conjecture for $E$.  
\begin{thm}[Kato, Kobayashi, Skinner-Urban, Wan]
\label{thm:mainconjistrue}
In the setting  of Theorem~\ref{thm:main}, Conjecture~\ref{ref:conjKatomain} holds true in the following cases:
\begin{itemize}
\item[(a)] $E$ has good ordinary reduction at $p$.
\item[(b)] $E$ has good supersingular reduction at $p$ and $N$ is square-free.
\item[(c)] $E$ has multiplicative reduction at $p$ and there exists a prime $\ell \mid\mid N$ such that there $\overline{\rho}_E$ is ramified at $\ell$.
\end{itemize}
\end{thm}
\begin{proof}
In the setting of (a), see the works of Kato and Skinner-Urban \cite{ka1,skinnerurban} (as well as the enhancement of the latter due to Wan~\cite{xinwanSigmaHilbert}, that lifts certain local hypothesis of Skinner and Urban) and in the situation of (b), the works of Kobayashi and Wan~\cite{kobayashi2003,xinwansupersingularmainconj}. In the setting of (c), the works of Skinner and Kato~\cite{skinner,ka1} yields the desired conclusion. Note that Kato has stated his divisibility result towards Conjecture~\ref{ref:conjKatomain} only when $E$ has good ordinary reduction at $p$. We refer the reader to \cite{rubinkatodurham} for the slightly more general version of his theorem required to treat the non-crystalline semistable case as well.
\end{proof}
We are now ready present a proof of Theorem~\ref{thm:main}, which we deduce as a direct consequence of Theorem~\ref{thm:mainconjistrue}.
\begin{proof}[Proof of Theorem~\ref{thm:main}]
The natural map
$$H^1_{\FF_{\LL}^*}(\QQ,\TT^*)^\vee/(\gamma-1)\cdot H^1_{\FF_{\LL}^*}(\QQ,\TT^*)^\vee\lra H^1_{\FFc^*}(\QQ,T^*)^\vee$$
has finite kernel and cokernel by the proof of Proposition~5.3.14 of \cite{mr02} (applied with the height one prime $\frak{P}=(\gamma-1)$). We note that the requirement on the $\LL$-module $H^2(\QQ_{S}/\QQ,\TT)[\gamma-1]$ is not necessary for the portion of this proposition concerning us. We further remark that since $H^0(\QQ_p,T^*)$ is finite, it follows that $H^2(\QQ_p,\TT)[\gamma-1]$ is pseudo-null and the proposition indeed applies. It follows from Proposition~\ref{prop:strisfinite} that $(\gamma-1)$ is prime to the characteristic ideal of $H^1_{\FF_{\LL}^*}(\QQ,\TT^*)^\vee$, and by Theorem~\ref{thm:mainconjistrue} that it is also prime to $\textup{char}\left(H^1(\QQ,\TT)/\LL\cdot\BBK_1\right)$. This tells us that 
$$\left\{c_{\QQ_n}^\BK\right\}=\BBK_1 \notin (\gamma-1)H^1(\QQ,\TT)=\ker\left(H^1(\QQ,\TT)\lra H^1(\QQ,T)\right)\,,$$
which amounts to saying that $c_{\QQ}^{\BK}\neq 0$. Furthermore, our running hypothesis that $r_{\textup{an}}=1$ together with Kato's explicit reciprocity laws implies that $c_{\QQ}^{\BK} \in H^1_f(\QQ,T)$. The proof of Proposition~\ref{prop:strisfinite} now shows that $\res_p(c_{\QQ}^{\BK})\neq 0$ (as otherwise, the module $H^1_{\FF_{\textup{str}}}(\QQ,T)$ would have been non-zero).
\end{proof}

\subsection{Logarithms of Heegner points and Beilinson-Kato classes}
\label{sec:heegner}
As before, we suppose that $E$ is an elliptic curve such that the residual representation $\overline{\rho}_E$ is surjective. Assume further that $p$ does not divide $\textup{ord}_\ell(j(E))$ whenever $\ell \mid N$ is a prime of split multiplicative reduction. Assume also that one of the following conditions hold true.
\begin{itemize}
\item[(a)] $E$ has good ordinary reduction at $p$.
\item[(b)] $E$ has good supersingular reduction at $p$ and $N$ is square-free.
\item[(c)] $E$ has multiplicative reduction at $p$ and there exists a prime $\ell \mid\mid N$ such that there $\overline{\rho}_E$ is ramified at $\ell$.
\end{itemize} 
With this set up, we will be able refine the conclusion of Theorem~\ref{thm:main} in a variety of cases and deduce that the square of the logarithm of a suitable Heegner point agrees with the logarithm of the Beilinson-Kato class $\BK_1$ up to an explicit non-zero algebraic factor. In these cases, we will therefore justify some of the hypothetical conclusions in \cite[\S3.3.3]{pr93grenoble} (see also Footnote 1).

We fix a Weierstrass minimal model $\mathcal{E}_{/\ZZ}$ of $E$. Let $\omega_\mathcal{E}$  be a N\'eron  differential that is normalized as in \cite[\S 3.4]{PR-ast} and is such that we have $\Omega_E^+:=\int_{E(\mathbb{C})^+}\omega_\mathcal{E}>0$ for the real period $\Omega_E^+$. Suppose till the end of this Introduction that $L(E,s)$ has a simple zero at $s=1$. In this situation, $E(\QQ)$ has rank one and the N\'eron-Tate height $ \langle P,P \rangle_{\infty}$ of any generator $P$ of the free part of $E(\QQ)$ is related via the Gross-Zagier theorem to the first derivative of $L(E,s)$ at $s=1$:
\be
\label{eqn:GZ}
\frac{L^\prime(E,1)}{\Omega_E^+}=C(E)\cdot \langle P,P \rangle_{\infty}
\ee 
with $C(E)\in \QQ^\times$.

Let $D_{\textup{cris}}(V)$ be the crystalline Dieudonn\'e module of $V$ and we define the element $\omega_{\textup{cris}}\in D_{\textup{cris}}(V)$ as that corresponds $\omega_\mathcal{E}$ under the comparison isomorphism. We let $\mathcal{H}\subset \LL\otimes_{\ZZ_p}\QQ_p$ denote Perrin-Riou's ring of distributions and let 
$$\frak{Log}_V: H^1(\QQ_p,\TT)\otimes_{\LL}\mathcal{H}\lra \mathcal{H}\otimes D_{\textup{cris}}(V)$$
be Perrin-Riou's extended dual exponential map and write $\al_{\BK}$ as a shorthand for the element $\frak{Log}_V\left(\textup{res}_p(\BBK_1)\right) \in \mathcal{H}\otimes D_{\textup{cris}}(V)$. We also let 
$$\log_V: H^1_f(\QQ_p,V)\stackrel{\sim}{\lra} D_{\textup{dR}}(V)/\textup{Fil}^0\, D_{\textup{dR}}(V)$$ denote Bloch-Kato logarithm.

\subsubsection{$E$ has good reduction at $p$}
\label{subsubsec:Ehasgoodreductionatp}
 In this case, $D_{\textup{cris}}(V)$ is a two dimensional vector space. Let $\alpha^{-1},\beta^{-1} \in \overline{\QQ}_p$ be the eigenvalues of the crystalline Frobenius $\varphi$ acting on $D_{\textup{cris}}(V)$. Extending the base field if necessary, let $D_\alpha$ and $D_\beta$ denote corresponding eigenspaces. Set $\omega_{\textup{cris}}=\omega_\alpha+\omega_\beta$ with $\omega_\alpha \in D_\alpha$ and $\omega_\beta\in D_\beta$. On projecting $\al_{\BK}$ onto either of these vector spaces we obtain $\al_{\BK,?} \in \mathcal{H}$ (so that $\al_{\BK,?}\cdot\omega_?$ is the projection of $\al_{\BK}$ onto $\mathcal{H}\otimes D_?$) for $?=\alpha,\beta$.

\begin{thm}[Kato, Kobayashi, Perrin-Riou]
\label{thm:mainsupersingularexplicit}
Suppose that $E$ has good supersingular reduction at $p$ and $N$ is square-free. For $?=\alpha,\beta$, 
\begin{itemize}
\item[(i)] the Amice transform of the distribution $\al_{\BK,?}$ is the Manin-Vishik, Amice-Velu $p$-adic $L$-function $L_{p,?}(E/\QQ,s)$ associated to the pair $(E,D_?)$; 
\item[(ii)] when $r_{\textup{an}}=1$, one of the two $p$-adic $L$-functions vanishes at $s=1$ to degree $1$;  
\item[(iii)] still when $r_{\textup{an}}= 1$, at least one of the associated $p$-adic height pairings $\langle\,,\,\rangle_{p,?}$ is non-degenerate,
\item[(iv)] $\log_E\left(\textup{res}_p(\BK_1)\right)=-(1-1/\alpha)(1-1/\beta)\cdot C(E)\cdot\log_E\left(\textup{res}_p(P)\right)^2$\,.
\end{itemize}
\end{thm} 
Note that the quantity $(1-1/\alpha)(1-1/\beta)=(1+1/p)$ belongs to $\QQ^\times$.

\begin{rem}
Kobayashi has in fact proved in \cite[Corollaries 1.3(ii) and 4.9]{kobayashiGZ} that when $r_{\textup{an}}=1$, \emph{both}  $p$-adic $L$-functions vanish at $s=1$ to degree $1$ and \emph{both} of the associated $p$-adic height pairings $\langle\,,\,\rangle_{p,?}$ are non-degenerate. The reason why we recall his results in this weaker form is that we hope to apply the strategy here to prove Theorem~\ref{thm:mainsupersingularexplicit}(iv) also in the case when $E$ has good ordinary reduction at $p$. In that case, if we had a $p$-adic Gross-Zagier formula for the critical slope $p$-adic $L$-function, we could have had proceeded precisely in this manner, even though we do not know that  both $p$-adic height pairings in this set up are non-trivial. See Remark~\ref{rem:mainthmgoodordinary} below for a further discussion concerning this point.
\end{rem}
\begin{proof}[Proof of Theorem~\ref{thm:mainsupersingularexplicit}]
The first assertion is due to Kato\footnote{This is accomplished after suitably normalizing $\BBK_1$ and throughout this work, we implicitly assume that we have done so.}, see \cite{ka1}. It follows from \cite[Proposition 2.2.2]{pr93grenoble} and Theorem~\ref{thm:main} that $\al_{\BK}^\prime(\mathds{1})\neq 0$. Thence, for at least one of $\alpha$ or $\beta$ (say it is $\alpha$) we have
$$\textup{ord}_{s=1} L_{p,\alpha}(E/\QQ,s)=1\,.$$
This completes the proof of the second assertion. The third follows from the $p$-adic Gross-Zagier formula of Kobayashi in~\cite{kobayashiGZ}. We will deduce\footnote{
Although it might be possible to prove (iv) relying only on the technology available in \cite{pr93grenoble}, we decided that we will stick to the current exposition (that benefits from the recent developments in the theory of triangulordinary Selmer groups) as our later discussion that concerns higher dimensional motives is in line with the approach we present here.} the last portion from the discussion in \cite[\S 3.3.3]{pr93grenoble}, as enhanced by Benois~\cite{benoiscomplex, benoisheights} building the work of Pottharst, combined with Kobayashi's $p$-adic Gross-Zagier formula.

Let $\omega_{\textup{cris}}$ and $\omega_\alpha,\omega_\beta$ be as above. Let us write $\frak{Log}_{V,\alpha} \in \textup{Hom}_{\LL}(H^1(\QQ,\TT),\mathcal{H})$ for the $\omega_\alpha$-coordinate of the projection of the image of $\frak{Log}_{V}$ parallel to $D_\beta$. We likewise define $\frak{Log}_{V,\beta}$; note that 
$$\frak{Log}_{V}=\frak{Log}_{V,\alpha}\cdot \omega_\alpha+\frak{Log}_{V,\beta}\cdot\omega_\beta\,.$$
Let $\omega^* \in D_{\textup{cris}}(V)/\textup{Fil}^0D_{\textup{cris}}(V)$ denote the unique element such that $[\omega,\omega^*]=1$, where 
$$[\,,\,]\,: D_{\textup{cris}}(V)\otimes D_{\textup{cris}}(V) \lra \QQ_p$$
the canonical pairing. Define  $\log_{\omega^*}(\BK_1)$ so that 
$$\log_{V}(\res_p(\BK_1))=\log_{\omega^*}(\res_p(\BK_1))\cdot\omega^*\,;$$
note that with our choice of $\omega$, we have $\log_{\omega^*}=\log_E$. The dual basis of $\{\omega_\alpha,\omega_\beta\}$ with respect to the pairing $[\,,\,]$
is $\{\omega_\beta^*,\omega_\alpha^*\}$, where $\omega_\beta^*$ (respectively, $\omega_\alpha^*$) is the image of $\omega^*$ under the inverse of the isomorphism $s_{D_\beta}:\,D_\beta\stackrel{\sim}{\ra} D_{\textup{cris}}(V)/\textup{Fil}^0D_{\textup{cris}}(V)$ (respectively, under the inverse of $s_{D_\alpha}$). Suppose that $\alpha$ is such that $\langle\,,\,\rangle_{p,\alpha}$ is non-degenerate. Then,
\begin{align}
\notag(1-1/\alpha)^2\cdot C(E)\cdot \langle P,P\rangle_{p,\alpha}&=\frac{d\,L_{p,\alpha}(E/\QQ,s)}{ds}\Big{|}_{s=1}\\
\label{eqn:BKcalculation2}&=\frak{Log}_{V,\alpha}(\partial_\alpha\mathbb{BK}_1)(\mathds{1})\\
\notag&=\left[\exp^*(\partial_\alpha{\BK}_1),(1-p^{-1}\varphi^{-1})(1-\varphi)^{-1}\cdot\omega_\beta^*\right]\\
\notag&=(1-p^{-1}\beta)(1-1/\beta)^{-1}\left[\exp^*(\partial_\alpha{\BK}_1),\omega^*\right]\\
\notag&=(1-1/\alpha)(1-1/\beta)^{-1}\frac{\left[\exp^*(\partial_\alpha{\BK}_1),\log_V(\res_p(\BK_1))\right]}{\log_{\omega^*}(\res_p(\BK_1))}\\
\label{eqn:BKcalculation9}&=-(1-1/\alpha)(1-1/\beta)^{-1}\frac{\langle\BK_1,\BK_1\rangle_{p,\alpha}}{\log_{\omega^*}(\res_p(\BK_1))}
\end{align}
where the first equality is Kobayashi's formula and the second will follow from the definition of $\frak{Log}_{V,\alpha}$ and the fact that it maps to Beilinson-Kato class to the Amice-Velu, Manin-Vishik distribution as soon as we define the \emph{derivative} $\partial\mathbb{BK}_1$ of the Beilinson-Kato class; we take care of this at the very end. The third equality will also from the explicit reciprocity laws of Perrin-Riou (as proved by Colmez) (c.f., the discussion in \cite[Section 2.1]{kbbleiintegralMC}) once we define projection $\partial{\BK}_1\in H^1_s(\QQ_p,V)$ of the derived Beilinson-Kato class $\partial\mathbb{BK}_1$. Fourth and fifth equalities follow from definitions (and using the fact that $\alpha\beta=p$). We now explain (\ref{eqn:BKcalculation9}) together with the definitions of the objects $\partial\mathbb{BK}_1$ and $\partial{\BK}_1$.

Let $D_{\textup{rig}}^\dagger(V)$ denote the $(\varphi,\Gamma)$-module attached to $V$ and $\mathbb{D}_\alpha \subset D_{\textup{rig}}^\dagger(V)$ the saturated $(\varphi,\Gamma)$-submodule of $D_{\textup{rig}} ^\dagger(V)$ attached to $D_\alpha$ by Berger. Set $\widetilde{\mathbb{D}}_\alpha:=D_{\textup{rig}}^\dagger(V)/\mathbb{D}_\alpha$, which is also a $(\varphi,\Gamma)$-module of rank one. 

Given a $(\varphi,\Gamma)$-module $\mathbb{D}$, one may define the cohomology $H^1(\mathbb{D})$ (respectively, Iwasawa cohomology $H^1_\Iw(\mathbb{D})$) group of $\mathbb{D}$ making use of the Fontaine-Herr complex associated to $\mathbb{D}$ (c.f., \cite[Section 1.2]{benoiscomplex}). A well-known result of Herr yields canonical isomorphisms
$$H^1_\Iw({D}_\textup{rig}^\dagger(V))\cong H^1 (\QQ_p,\TT)\otimes\mathcal{H}\,\,,\,\, H^1({D}_\textup{rig}^\dagger(V))\cong H^1 (\QQ_p,V)\,.$$
Furthermore, we have an exact sequence 
$$0\ra H^1_\Iw(\mathbb{D}_\alpha) \lra  H^1_\Iw(D_{\textup{rig}}^{\dagger}(V))\stackrel{{\pi_{/f}}}{\lra} H^1_\Iw(\widetilde{\mathbb{D}}_\alpha)$$
and we set $\textup{Res}_p^{/f}(\mathbb{BK}_1):=\pi_{/f}\circ \res_p(\mathbb{BK}_1) \in H^1_\Iw(\widetilde{\mathbb{D}}_\alpha)$. It turns also that we have the following commutative diagram:
$$\xymatrix{\ar[d]  H^1_\Iw(D_{\textup{rig}}^{\dagger}(V))\ar[r]^{{\pi_{/f}}}\ar[d]_{\textup{pr}_0}& H^1_\Iw(\widetilde{\mathbb{D}}_\alpha)\ar[d]^{\textup{pr}_0}\\
H^1(\QQ_p,V)\ar[r]& H^1_{/f}(\QQ_p,V)}$$
so that $\textup{pr}_0\circ\textup{Res}_p^{/f}(\mathbb{BK}_1)=\res_p^{/f}(\BK_1)=0$. Here, $H^1_{/f}(\QQ_p,V):=H^1(\QQ_p,V)/H^1_{f}(\QQ_p,V)$  is the singular quotient. The vanishing of $\textup{pr}_0\circ\textup{Res}_p^{/f}(\mathbb{BK}_1)$ shows that 
$$\textup{Res}_p^{/f}(\mathbb{BK}_1) \in \ker(\textup{pr}_0:H^1_\Iw(\widetilde{\mathbb{D}}_\alpha)\ra  H^1_{/f}(\QQ_p,V))=(\gamma-1)\cdot H^1_\Iw(\widetilde{\mathbb{D}}_\alpha),$$
so that there is an element $\partial\mathbb{BK}_1 \in H^1_\Iw(\widetilde{\mathbb{D}}_\alpha)$ with the property that 
$$\log_p\chi_{\cyc}(\gamma)\cdot\textup{Res}_p^{/f}\left(\mathbb{BK}_1\right)=(\gamma-1)\cdot\partial\mathbb{BK}_1\,,$$
and as a matter of fact, this element is uniquely determined in our set up. Berger's reinterpretation of the Perrin-Riou map $\frak{Log}_{V,\alpha}$ shows that it factors through $\pi_{/f}$ and therefore also that
$$\frak{Log}_{V,\alpha}(\partial_\alpha\mathbb{BK}_1)(\mathds{1})=\frac{d}{ds}\,L_{p,\alpha}(E/\QQ,s)\Big{|}_{s=1}$$
as we have claimed in (\ref{eqn:BKcalculation2}). The element $\partial\BK_1 \in H^1(\QQ_p,V)$ is simply $\textup{pr}_0(\partial\mathbb{BK}_1)$.

Also attached to $\mathbb{D}_\alpha$, one may construct an extended Selmer group $\mathbf{R}^1\Gamma(V,\mathbb{D}_\alpha)$, which is the cohomology of a Selmer complex $\mathbf{R}\Gamma(V,\mathbb{D}_\alpha)$ (c.f., \cite[Section 2.3]{benoiscomplex}). It follows from \cite[Proposition 11]{benoiscomplex} that this Selmer group agrees in our set up with the classical Bloch-Kato Selmer group. It comes equipped with a $p$-adic height pairing (Section 4.2 of loc. cit.) and as a matter of fact, this height pairing agrees with $\langle\,,\,\rangle_{p,\alpha}$ by \cite[Theorem 5.2.2]{benoisheights} and (\ref{eqn:BKcalculation9}) follows from the Rubin-style formula proved in \cite[Theorem 4.13]{benoisbuyukboduk}.

Using now the fact that $\langle \cdot,\cdot \rangle_{p,\alpha}$ and $\log_{\omega^*}(\res_p\,(\,\cdot\,))^2$ are both non-trivial quadratic forms on  the one dimensional $\QQ_p$-vector space $E(\QQ)\otimes\QQ_p$ and combining with (\ref{eqn:BKcalculation9}), we conclude that 
$$\frac{\langle P,P\rangle_{p,\alpha}}{\log_{\omega^*}(P)^2}=\frac{\langle \BK_1,\BK_1\rangle_{p,\alpha}}{\log_{\omega^*}(\BK_1)^2}=-(1-1/\alpha)(1-1/\beta)\cdot C(E)\cdot \frac{\langle P,P\rangle_{p,\alpha}}{\log_{\omega^*}(\BK_1)}$$
\end{proof}
\begin{rem}
\label{rem:mainthmgoodordinary}
Much of what we have recorded above for a supersingular prime $p$ applies verbatim for a good ordinary prime as well. Suppose that $\alpha$ is the root of the Hecke polynomial which is a $p$-adic unit, so that $v_p(\beta)=1$. In this case, we again have two $p$-adic $L$-functions: The projection of $\al_{\BK}$ to $D_\alpha$ yields the Mazur-Tate-Teitelbaum $p$-adic $L$-function $L_{p,\alpha}(E/\QQ,s)$, whereas its projection to $D_\beta$ should agree\footnote{We remark that this conclusion does not formally follow directly from Kato's explicit reciprocity laws. See \cite{loefflerzerbes};  also \cite{hansen} where a proof of this was announced.} with the \emph{critical slope} $p$-adic $L$-function $L_{p,\beta}(E/\QQ,s)$ of Bella\"iche and Pollack-Stevens. The analogous statements to those in Theorem~\ref{thm:mainsupersingularexplicit} therefore reduces to check that one of the following holds true:
\begin{itemize}
\item[a)] There exists a $p$-adic Gross-Zagier formula for the critical slope $p$-adic $L$-function $L_{p,\beta}(E/\QQ,s)$, or that
\item[b)] $\textup{ord}_{s=1}L_{p,\beta}(f,s)\geq \textup{ord}_{s=1}L_{p,\alpha}(f,s)\,.$ 
\end{itemize}
We suspect that the latter statement may be studied through a \emph{critical slope} main conjecture and its relation with the ordinary main conjecture. We will pursue this direction in a future joint work with R. Pollack.
\end{rem}
\subsubsection{$E$ has non-split-multiplicative reduction at $p$} In this case, $D_{\textup{cris}}(V)$ is one-dimensional and we have $\al_{\BK}=\al_{\textup{MTT}}\cdot\omega_{\textup{cris}}$, where $\al_{\textup{MTT}}\in \LL$ is the Mazur-Tate-Teitelbaum measure. The following is a consequence of the $p$-adic Gross-Zagier formula (c.f., \cite{disegniGZ}), the Rubin-style formula proved in \cite[Proposition 11.3.15]{nek} and our main Theorem~\ref{thm:main}:
\begin{thm}[Disegni, Gross-Zagier, Nekov\'a\v{r}]
\label{thm:mainmultiplicative}
Suppose that Nekov\'a\v{r}'s $p$-adic height pairing associated to the canonical splitting of the Hodge-filtration on the semi-stable Dieudonn\'e module $D_{\textup{st}}(V)$ is non-vanishing. Then,
$$\log_V\left(\textup{res}_p(\BK_1)\right)\cdot\log_V\left(\textup{res}_p(P)\right)^{-2}\in \overline{\QQ}^\times\,.$$
\end{thm}  
\begin{rem}
\label{rem:modularforms}
Our argument so far easily adapt to prove that analogous conclusions hold true for a \emph{essentially self-dual} elliptic modular form $f$ which verifies the hypotheses of Remark~\ref{rem:selfdualmodforms} and for which the natural map
$$\textup{res}_p:H^1_f(\QQ,V_f)\ra H^1_f(\QQ_p,V_f)$$
is non-zero. Here, $V_f$ is the self-dual twist of Deligne's representation, as in Remark~\ref{rem:selfdualmodforms}.
\end{rem}

\section{Part II.  $\LL$-adic Kolyvagin systems, Beilinson-Flach elements, Coleman-Rubin-Stark elements and Heegner points}
\label{subsec:KS}
In this section, we recast the proof of Theorem~\ref{thm:main} in terms of the theory of $\LL$-adic Kolyvagin systems (and in great generality), with the hope that it will provide us with further insights to analyse the analogs of Perrin-Riou's predictions in other situations. 
\subsection{The set up}
Let $\mathscr{M}$ denote a self-dual motive over either a totally real or CM number field $K$ with coefficients in a number field $L$. Let $L(\mathscr{M},s)$ denote its $L$-function and write $r_{\textup{an}}(\mathscr{M})$ for its order of vanishing at the central critical point (which is a quantity conditional on the expected functional equation and analytic continuation). Let $V$ denote its $p$-adic realization (which a finite finite vector space over $E$ (a completion of $L$ at a prime above $p$) endowed with a continuous $G_K$-action) and $T\subset V$ a $\frak{o}_E$-lattice. We write $\overline{T}$ for $T/\mm_ET$, where $\mm_E$ is the maximal ideal of $\frak{o}_E$ and $T^*:\textup{Hom}(T,\bbmu_{p^\infty})$ for the Cartier dual of $T$. We will set $2r:=[K:\QQ]\dim_E V$ (note by the self-duality of $\mathscr{M}$ that the quantity on the right is indeed even) and let $S$ denote the finite set of places of $K$ which consists of all primes above $K$, all archimedean places and all places at which $V$ is ramified. We will assume throughout Section~\ref{subsec:KS} that $V$ is \emph{Panchishkin ordinary} in the sense that for each prime $\frak{p}$ of $K$ above $p$, the following three conditions hold true:
\begin{itemize} 
\item[(HP0)] There is a direct summand $F^+_{\frak{p}}T\subset T$ (as an $\frak{o}_E$-submodule) of rank $ r_{\frak{p}}$,  which is stable under the $G_{K_\frak{p}}$-action and such that ${\sum_{\frak{p}\mid p} r_{\frak{p}}=r\,.}$
\item[(HP1)] The Hodge-Tate weights of the subspace $F_\frak{p}^+V:=F_\frak{p}^+T\otimes E$ are  strictly negative.
\item[(HP2)]  The Hodge-Tate weights of the quotient $V/F_\frak{p}^+V$ are non-negative.
\end{itemize}
\begin{rem}
When $r=1$, we may in fact drop the Panchishkin condition on $T$ without any further work.  
\end{rem}
\begin{example}
\label{example:selfdualmotive}
Suppose that $A/K$ is an abelian variety of dimension $g$, which has good ordinary reduction at all primes above $p$. Then $r=g[K:\QQ]$ and the $p$-adic realization of the motive $h^1(A)(1)$ is $V_p(A):=T_p(A)\otimes \QQ_p$ is the $p$-adic Tate-module and $V_p(A)$ is Panchishkin ordinary. 
\end{example}

For positive integers $\alpha$ and $k$, we set $R_{k,\alpha}:=\LL/(p^k,(\gamma-1)^\alpha)$ and $T_{k,\alpha}:=\TT\otimes R_{k,\alpha}$. 

\subsection{Module of $\LL$-adic Kolyvagin systems}
\label{subsec:KSway}
Throughout this section, we shall assume that the hypotheses \textup{\textbf{(H.0)}}\,-\,\textup{\textbf{(H.3)}} of \cite[Section 3.5]{mr02} are in effect as well as the following two hypotheses:\\
\textbf{(H.nA)} $H^0(K_\frak{p},\overline{T})=0$ for every prime $\frak{p}$ of $K$ above $p$.
\\
\textbf{(H.Tam)} We have $H^0\left(G_{K_\lambda}/I_\lambda,H^0(I_\lambda,V/T)\big{/}H^0(I_\lambda,V/T)_\textup{div}\right)$ for every prime $\lambda \in S$ (where $I_\lambda \subset G_{K_\lambda}$ stands for the inertia subgroup).

\begin{define}
We define the Greenberg Selmer structure $\FF_\Gr$ by the local conditions
$$H^1_{\FF_{\Gr}}(K_\frak{p},\TT):=\textup{im}\left(H^1(K_{\frak{p}},F_\frak{p}^+T\otimes\LL)\ra H^1(K_{\frak{p}},\TT)\right)$$
for every prime $\frak{p}$ above $p$, and by setting $H^1_{\FF_{\Gr}}(K_\lambda,\TT):=H^1(K_\lambda,\TT)$ for $\lambda\nmid p$.
\end{define}
Under the hypothesis \textbf{(H.nA)} (and the self-duality assumption on $T$), it follows that $H^1(K_p,\TT)$ is a free $\LL$-module of rank $2r$ and 
$$H^1_{\FF_{\Gr}}(K_{p},\TT):=\oplus_{\frak{p}\mid p}H^1_{\FF_{\Gr}}(K_\frak{p},\TT)\subset H^1(K_{p},\TT)$$ 
is a direct summand of rank $r$. We fix a direct summand $H^1_+(K_p,\TT)\supset H^1_{\FF_{\Gr}}(K_{p},\TT)$ of rank $r+1$. 

We will write $H^1_f(K_p,T)$ for the image of $H^1_{\FF_{\Gr}}(K_p,\TT)$ under $\textup{pr}_0$ and also denote the Selmer group determined by the propagation of the Selmer structure $\FF_\Gr$ to $T$ by $H^1_f(K,T)$. Let $H^1_+(K_p,T)$ denote the image of $H^1_+(K_p,\TT)$ under $\textup{pr}_0$. The $\frak{o}_E$-module $H^1_+(K_p,T)$ is a direct summand of $H^1(K_p,T)$ of rank $r+1$ (thanks to our hypothesis {\textup{\textbf{H.nA}}}). Using the fact that $\FF_\Gr$ is self-dual, this determines (via local Tate duality) a direct summand $H^1_-(K_p,T)\subset H^1_{\FF_{\Gr}}(K_p,T)$ of rank $r-1$ and allows us to define a Selmer structure $\FF_-$ on $T$ (which is given by the local conditions determined by $H^1_-(K_p,T)\subset H^1(K_p,T)$ at $p$, and by propagating of $\FF_\Gr$ at any other place).

\begin{define}
We define the $\FF_+$ by the local conditions
$$H^1_{\FF_{+}}(K_p,\TT)=H^1_+(K_p,\TT)$$
and by setting $H^1_{\FF_{+}}(K_\lambda,\TT):=H^1(K_\lambda,\TT)$ for $\lambda\nmid p$.
\end{define}
Given a Selmer structure $\FF$ on $\TT$, we let $\chi(\FF,\TT):=\chi(\FF,\overline{T})$ denote the core Selmer rank (in the sense of \cite[Definition 4.1.11]{mr02}) of the propagation of the Selmer structure $\FF$ to $\overline{T}$. It follows from the discussion in \cite[\S4.1 and \S5.2]{mr02} that $\chi(\FF_\LL,\TT)=r$.
\begin{prop}
\label{prop:genericcoreselmerrank}
We have $\chi(\FF_\Gr,\TT)=0$; whereas $\chi(\FF_+,\TT)=1$.
\end{prop}
\begin{proof}
These follow using global Euler characteristic formulae, together with the fact that $\chi(\FF_\LL,\TT)=r$.
\end{proof}
\begin{define}
\label{def:kolprimestransversecond}
Given positive integers $\alpha$ and $k$, we let $\PP_{k,\alpha}$ (respectively, $\PP_j$)  denote the set of \emph{Kolyvagin primes} for $T_{k,\alpha}$, as introduced in \cite[Section 2.4]{kbb}. We set $\PP:=\PP_{1,1}$.
\end{define}
Given an integer $j\geq k+\alpha$, one may define the module $\KS(\FF_+,T_{k,\alpha},\PP_j)$ of Kolyvagin systems for the Selmer structure $\FF_+$ on the artinian module $T_{k,\alpha}$ as in the paragraph following Definition 3.3 in \cite{kbb} (after replacing the Selmer structure $\FFc$ by our more general Selmer structure $\FF_+$). 
\begin{define}
The $\LL$-module
$$\KS(\FF_+,\TT,\PP):=\varprojlim_{k,\alpha}\left(\varinjlim_{j\geq k+\alpha}\KS(\FF_+,T_{k,\alpha},\PP_j)\right) $$
is called the module of $\LL$-adic Kolyvagin systems.
\end{define}
\begin{thm}
\label{thm:mainkbb}
Under the hypotheses \textup{\textbf{(H.0)}}\,-\,\textup{\textbf{(H.3)}} of \cite[Section 3.5]{mr02} and assuming \textup{\textbf{(H.Tam)}} and \textup{\textbf{(H.nA)}}, the natural map $\KS(\FF_+,\TT,\PP)\ra \KS(\FF_+,\overline{T},\PP)$ is surjective, the $\LL$-module $\KS(\FF_+,\TT,\PP)$ is free of rank one and its generated by any $\LL$-adic Kolyvagin system whose projection to $\KS(\FF_+,\overline{T},\PP)$ is non-zero.
\end{thm}
\begin{proof}
This is the main theorem of \cite{kbb}; one only needs to replace $\FFc$ in loc. cit. with $\FF_+$ (but that entails no complications).
\end{proof}
\begin{rem}
\label{rem:whathappenswhenpnonanomalous}
When $T=T_p(E)$  is the $p$-adic Tate module of an elliptic curve $E/\QQ$, the hypothesis \textup{\textbf{(H.Tam)}} is equivalent to the requirement that $p$ does not divide $\textup{ord}_\ell(j(E))$ whenever $\ell$ is a prime of split multiplicative reduction. 

Still when $T=T_p(E)$, the hypothesis that $H^2(\QQ_p,\overline{T})=0$  is equivalent to the requirement that $E(\QQ_p)[p]=0$. This is precisely the condition that the prime $p$ be \emph{non-anomalous} for $E$ in the sense of \cite{mazur-anom}. It is easy to see that all primes $p>5$ at which 
\begin{itemize} 
\item either $E$ has supersingular reduction, 
\item or non-split-multiplicative reduction, 
\item or split-multiplicative reduction with $p>7$, 
\item or good ordinary reduction with $a_p(E)\neq 1$,
\item or for elliptic curves which possess a non-trivial $\QQ$-rational torsion
\end{itemize}
are non-anomalous.

In our supplementary note \cite[Appendix A]{kbbArxiv1}, we are able to lift the non-anomaly hypothesis on $T_p(E)$ and prove the following in this setting:

\begin{thm}
\label{thm:mainKS} Suppose that $E$ is an elliptic curve such that the residual representation 
$$\overline{\rho}_E:\,G_{\QQ,S}\lra \textup{Aut}(E[p])$$
is surjective. Assume further that $p$ does not divide $\textup{ord}_\ell(j(E))$ whenever $\ell \mid N$ is a prime of split multiplicative reduction. Then,
\\\\
{\upshape{\textbf{i)}}} the $\LL$-module $\KS(\TT_p(E))$ of $\LL$-adic Kolyvagin systems contains a free $\LL$-module of rank one with finite index;
\\
{\upshape{\textbf{ii)}}} there exists a $\LL$-adic Kolyvagin system $\pmb{\kappa} \in \KS(\TT_p(E))$ with the property that $\textup{pr}_0(\pmb{\kappa}) \in \KS(T_p(E))$ is non-zero.
\end{thm}
\end{rem}
Let $\res_{+/f}$ denote the compositum of the arrows
$$H^1_{\FF_+}(K,\TT)\lra H^1_{+}(K_p,\TT)\lra {H^1_{+}(K_p,\TT)}\big{/}{H^1_{\FF_\Gr}(K_p,\TT)}\,.$$
\begin{define}
Given a $\LL$-adic Kolyvagin system $\bbkappa \in \KS(\FF_+,\TT,\PP)$ for which $\bbkappa_1\neq 0$, we define its \emph{defect} by setting 
$$\delta(\bbkappa):=\textup{char}\left(H^1(K,\TT)/\LL\cdot\bbkappa_1\right)\big{/}\textup{char}\left(H^1_{\FF_+^*}(K,\TT^*)^\vee\right)\,.$$
Observe that $\delta(\bbkappa)\subset \LL$ thanks to the general Kolyvagin system machinery. Whenever $\res_{+/f}\left(\bbkappa_1\right)\neq 0$, we also define the \emph{Kolyvagin constructed $p$-adic $L$-function} 
$$\frak{L}_p(\bbkappa):=\textup{char}\left(H^1_{+/f}(K_p,\TT)\big{/}\res_{+/f}(\bbkappa_1)\right)\,.$$ 
Note in this case (using Poitou-Tate global duality) that the defect of $\bbkappa$ is given as the quotient 
$\delta(\bbkappa)=\frak{L}_p(\bbkappa)/\textup{char}(H^1_{\FF_{\textup{Gr}}^*}(K,\TT^*)^\vee)\,.$
\end{define}
\begin{rem}
\label{rem:defectcouldbesmall}
It follows from Theorem~\ref{thm:mainkbb} that any generator of the module $\KS(\FF_+,\TT,\PP)$ of $\LL$-adic Kolyvagin systems is $\LL$-primitive (in the sense of \cite[Definition 5.3.9]{mr02}). It then follows from \cite[5.3.10(iii)]{mr02} that for any generator $\bbkappa$ of $\KS(\FF_+,\TT,\PP)$ we have $\delta(\bbkappa)=\LL$.
\end{rem}
\begin{example}
\label{example:ellipticcurve}
Suppose $E/\QQ$ is an elliptic curve and $T=T_p(E)$. Let $\bbkappa^\BK \in \KS(\FF_\LL,\TT)$ denote the $\LL$-adic Kolyvagin system associated to $E$, that descend from the Beilinson-Kato elements via \cite[Theorem 5.3.3]{mr02}. It follows from Theorem~\ref{thm:mainconjistrue} (in all cases that it applies), $\delta(\bbkappa^{\kappa})$ is generated by a power of $p$ and in particular, it is prime to $(\gamma-1)$. 
\end{example}
Recall the direct summand $H^1_-(K_p,T)\subset H^1_{\FF_{\Gr}}(K_p,T)$ of rank $r-1$ and define the map
$\res_{f/-}$ as the compositum of the arrows
$$H^1_f(K,T)\lra H^1_f(K_p,T)\lra H^1_f(K_p,T)\big{/}H^1_-(K_p,T)\,.$$
\begin{thm}
\label{thm:PRwithKS} Assume that  \textup{\textbf{(H.Tam)}}, \textup{\textbf{(H.nA)}} and the hypotheses \textup{\textbf{(H.0)}}\,-\,\textup{\textbf{(H.3)}} of Mazur and Rubin in \cite{mr02} hold true.\\
\textbf{\textup{i)}} Suppose that the map $\res_{f/-}: H^1_f(K,T)\ra H^1_{f/-}(K_p,T)$
is injective.
\begin{itemize}
\item[(a)] For any non-trivial  $\bbkappa  \in \KS(\FF_+,\TT,\PP)$, we have $\bbkappa_1\neq 0$. 
\item[(b)] For any $\bbkappa$ whose defect $\delta(\bbkappa)$ is prime to $(\gamma-1)$, we have $\kappa_1\neq 0$.
\end{itemize}
\textbf{\textup{ii)}} Conversely, if $\kappa_1\neq 0$ for some $\kappa \in\KS(\FF_+,T,\PP)$, then the map $\res_{f/-}$ is injective.
\end{thm}
See the discussion in Remark~\ref{rem:PRconjforEbyKS} and Conjecture~\ref{conj:analyticrankvslocalizationmapinjective} pertaining to the injectivity of the map $\res_{f/-}$.
\begin{proof}
The requirement that the map $\res_{f/-}$ be injective is equivalent to asking that $H^1_{\FF_-}(K,T)=0$. The proof of Theorem~\ref{thm:mainonlyif} adapts without difficulty (on replacing $\FF_{\textup{str}}$ with $\FF_-$ and $\FFc^*$ with $\FF_+^*$) to show that $H^1_{\FF_+^*}(K,T^*)$ is finite. It follows from \cite[Corollary 5.2.13(i)]{mr02} that for every non-trivial Kolyvagin system ${\kappa}\in \KS(\FF_+,T,\PP)$, we have $\kappa_1\neq 0$ for its initial term. Theorem~\ref{thm:mainkbb} shows that ${\kappa}$ lifts to a $\LL$-adic Kolyvagin system $\bbkappa \in \KS(\FF_+,\TT,\PP)$, and we have $\bbkappa_1\neq0$ (as we have $\textup{pr}_0(\bbkappa_1)=\kappa_1\neq 0$). Since the $\LL$-module $\KS(\FF_+,\TT,\PP)$ is cyclic and $H^1(K,\TT)$ is torsin-free under our running assumptions, (a) follows. 

Let now $\bbchi$ be a $\LL$-primitive Kolyvagin system (such $\bbchi$ exists and generates the module $\KS(\FF_+,\TT,\PP)$ by Theorem~\ref{thm:mainkbb}). Let $g \in \LL$ be such that $\bbkappa=g\cdot\bbchi$. It follows from \cite[Theorem 5.3.10(iii)]{mr02} and the choice of $\bbkappa$ that $g$ is prime to $(\gamma-1)$. Furthermore, since $\textup{pr}_0(\bbchi_1)\neq 0$ by the discussion above, it follows that 
$\kappa_1=\textup{pr}_0(g)\cdot \textup{pr}_0(\bbchi_1)\neq 0$. This completes the proof of (b).

To prove (ii), note that $H^1_{\FF_+^*}(K,T^*)$ is finite by our assumption and \cite[Theorem 5.2.2]{mr02}. The proof of Theorem~\ref{thm:mainonlyif} shows (after suitable alterations, as we have pointed out in the first paragraph of this proof) that this implies the vanishing of $H^1_{\FF_-}(K,T)$. This is precisely what we desired to prove.
\end{proof}
\begin{rem}
\label{rem:PRconjforEbyKS}
For $T=T_p(E)$ for an elliptic curve as in Example~\ref{example:ellipticcurve} above, the $\res_{f/-}$ in the statement of Theorem~\ref{thm:PRwithKS} is simply the localization map at $p$. It is easy  to see using the work of Gross-Zagier, Kolyvagin and Skinner (under additional mild hypothesis) that this map is injective if and only if $r_{\textup{an}}(E)=1$. In particular, Perrin-Riou's conjecture in this set up follows from Theorem~\ref{thm:PRwithKS} and the discussion in Example~\ref{example:ellipticcurve}.
\end{rem}
The discussion in Remark~\ref{rem:PRconjforEbyKS} leads us to the following prediction:
\begin{conj}
\label{conj:analyticrankvslocalizationmapinjective}
There exists a choice of the direct summand $H^1_+(K_p,\TT)$ such that the map $\res_{f/-}$ is injective iff $r_{\textup{an}}(\mathscr{M})\leq 1$.
\end{conj}
\subsection{Example: Perrin-Riou's conjecture for Beilinson-Flach elements}
Let $E/\QQ$ be a non-CM elliptic curve with conductor $N$ and let $f_E \in S_2(\Gamma_0(N))$ denote the associated newform. Assume that the residual representation $\overline{\rho}_E:G_K\ra \textup{Aut}(E[p])$ is absolutely irreducible and suppose that $E$ has good ordinary reduction at the prime $p$. Fix an embedding $\iota_p:\overline{\QQ}\ra \mathbb{C}_p$. Suppose $K$ is an imaginary quadratic extension of $\QQ$ that satisfies the weak Heegner hypothesis for $E$, so that the order of vanishing $r_{\textup{an}}(E/K):=\textup{ord}_{s=1}L(E/K,s)$ is odd. Suppose further that $p$ does not divide $\textup{ord}_\ell(j(E))$ whenever $\ell | N$ is a prime of split multiplicative reduction. 

We will assume throughout this subsection that the prime $p$ splits in $K$ and write $p=\wp\wp^c$ as a product of primes of $K$, where the prime $\wp$ is induced from $\iota_p$. Fix forever an auxiliary modulus $\frak{f}$ of $K$ which is prime to $p$ and the ray class group of $K$ modulo $\frak{f}$ is prime to $p$. Fix also a ring class character $\alpha$ modulo $\frak{f}p^\infty$ of finite order, for which we have $\alpha(\wp)\neq \alpha(\wp^c)$. We let $\frak{o}$ denote the finite flat extension of $\ZZ_p$ in which $\alpha$ takes its values and write $T=T_p(E)\otimes \alpha^{-1}$ for the free $\frak{o}$-module of rank $2$ on which $G_K$ acts diagonally. Set $\mathscr{D}(T):=\textup{Hom}(T,\frak{o})(1)\cong T_p(E)\otimes\alpha$. Let $F^+T\subset T$ denote the Greenberg subspaces of $T$; this is a free $\frak{o}$-module of rank one such that the $G_{\QQ_p}$-action on the quotient $T/F^+T$ is unramified. 

In this set up, one may define the Greenberg Selmer structure $\FF_{\Gr}$ on $\TT$ as above, as well as modify this Selmer structure appropriately to apply Theorem~\ref{thm:mainkbb}. 

\begin{define}
\label{ref:FplusforBF}
We define the Selmer structure $\FF_+$ on $\TT$ by relaxing the local conditions at the prime $\wp$.
\end{define}

It follows from the discussion in \cite[Section 3.2]{kbbleiBFpord} (with the choice $f_1=f_E$ in loc.cit.), Beilinson-Flach element Euler system of Lei-Loeffler-Zerbes~\cite{LLZ2} gives rise to a Kolyvagin system $\bbkappa^{\textup{BF}}=\{\kappa^{\textup{BF}}_\eta\}_\eta \in \KS(\FF_+,\TT)$ (where the indices run through square free ideals $\eta$ of $\frak{o}_K$ that are products of appropriately chosen Kolyvagin primes) which we call the \emph{Beilinson-Flach Kolyvagin system}. Then $\mathbb{BF}_1:=\kappa_1^{\textup{BF}}\in H^1_{\FF_+}(K,\TT)$ and we write $\textup{BF}_1 \in H^1_{\FF_+}(K,T)$ for its image. The explicit reciprocity laws for the Beilinson-Flach elements in \cite{KLZ2} show that $\textup{res}_\wp^{s}(\textup{BF}_1)\not= 0$ iff $r_{\textup{an}}(E,\alpha):=\textup{ord}_{s=1}L(E,\alpha.1)$ equals $0$. Here, 
$$\res_\wp^s: H^1(K,T)\lra H^1(K_\wp,T)/H^1_{\FF_\Gr}(K_\wp,T)$$
is the singular projection.

In this particular situation, the map $\res_{f/-}$ is simply the map 
$$\res_{f/-}=\res_\wp:\,H^1_f(K,T)\lra H^1_{\FF_{\Gr}}(K_\wp,T).$$
When $r_{\textup{an}}(E,\alpha)=0$, then $\textup{res}_\wp^{s}(\textup{BF}_1)\not= 0$ and the Kolyvagin system machinery shows that $H^1_f(K,T)=0$ and the $\res_\wp$ is injective for this trivial reason.
\begin{prop}
\label{prop:resisinjectiveBF}
If $r_{\textup{an}}(E,\alpha)=1$, then the map $\res_\wp:\,H^1_f(K,T)\ra H^1_{\FF_{\Gr}}(K_\wp,T)$ is injective.
\end{prop}
\begin{proof}
Let $V_p(E):=T_p(E)\otimes\QQ_p$ and $V:=T\otimes \QQ_p$. We let $M$ denote the finite abelian extension of $K$ that is cut by $\alpha$. Note that we have 
\begin{align*} 
H^1_f(K,V)\stackrel{\sim}{\lra}H^1_f(M,V)^{G_K}&\stackrel{\sim}{\lra}\\
& \left(H^1_f(M,V_p(E))\otimes\alpha^{-1}\right)^{G_K}\stackrel{\sim}{\lra}H^1_f(M,V_p(E))^\alpha
\end{align*}
where the first arrow follows from the inflation-restriction sequence and the rest are self-evident. It follows from \cite{bentwistedGZ} and (a mild extension of) the main results of \cite{benGL2} that $H^1_f(M,V_p(E))^\alpha$ is a $1$-dimensional $\textup{Frac}(\frak{o})=L$-vector space as well as that 
$$E(M)^{\alpha}\stackrel{\sim}{\lra} H^1_f(M,V_p(E))^\alpha$$
where for an abelian group $X$, we write $X^{\alpha}:=\left(X\,\hat{\otimes}\, L(\alpha^{-1})\right)^{G_K}$ with $L(\alpha^{-1})$ being the $1$-dimensional $L$-vector space on which $G_K$ acts by $\alpha^{-1}$.
We therefore have the following commutative diagram
$$\xymatrix{E(M)^\alpha\ar[d]^{\res_\wp}\ar[r]^(.4){\sim}&H^1_f(M,V_p(E))^\alpha\ar[r]^(.57){\sim}\ar[d]^{\res_\wp}& H^1_f(K,V)\ar[d]^{\res_\wp}\\
\left(\bigoplus_{\frak{p}\mid \wp} E(M_\frak{p})\right)^{\alpha}\ar[r]^(.5){\sim}&H^1_{\FF_{\Gr}}(M_\wp,V_p(E))^\alpha\ar[r]^(.57){\sim}&H^1_{\FF_{\Gr}}(K_\wp,V)
}$$
The left-most arrow is evidently injective (as its source is spanned by an $M$-rational point) and hence, the right-most arrow is injective as well.
\end{proof}
The following statement is the Perrin-Riou conjecture for Beilinson-Flach elements.
\begin{cor}
\label{cor:PRBFelement}
Assume that the residual representation $\overline{\rho}_E$ $($afforded by $E[p]$$)$ is absolutely irreducible. Suppose also that there exists a rational prime $q \in S$ which does not split in $K/\QQ$ and $\overline{\rho}_E$ is ramified at $q$. If $r_{\textup{an}}(E,\alpha)=1$, then $\res_\wp(\textup{BK}_1)$ is non-zero.
\end{cor}

\begin{proof}
The main theorem in \cite{wanpordinaryindefinite} (which applies in our set up) shows that the defect $\delta(\bbkappa^{\textup{BF}})$ of the Beilinson-Flach Kolyvagin system is prime to $(\gamma-1)$ and the proof follows from Theorem~\ref{thm:PRwithKS}(i) and Proposition~\ref{prop:resisinjectiveBF}.
\end{proof}

\subsection{CM Abelian Varieties and Perrin-Riou-Stark elements}
\label{subsec:CMabvarPRStark} Let $K$ be a CM field and $K_+$ its maximal totally real subfield that has degree $g$ over $\QQ$. Fix a complex conjugation $c \in \textup{Gal}(\overline{\QQ}/K_+)$ lifting the generator of $\textup{Gal}(K/K_+)$. Fix forever an odd prime $p$ unramified in $K/\QQ$ and an embedding $\iota_p: \overline{\QQ} \hookrightarrow \overline{\QQ}_p$. 
\subsubsection{CM types and $p$-ordinary abelian varities}
\label{subsubsec:CMtypesabvar}
Fix a \emph{$p$-ordinary CM-type} $\Sigma$; this means that the embeddings $\Sigma_p:=\{\iota_p\circ \sigma\}_{\sigma \in \Sigma}$ induce exactly half of the places of $K$ over $p$. Identify $\Sigma_p$ with the associated subset of primes $\{\frak{p}_1,\cdots,\frak{p}_s\}$ of $K$ above $p$ and $\Sigma_p^c=\{\frak{p}_1^c,\cdots\frak{p}_s^c\}$. Note that the disjoint union $\Sigma_p \sqcup \Sigma_p^c$ is the set of all primes of $K$ above $p$. It follows that there then exists an abelian variety that has CM by $K$ and has good ordinary reduction at $p$, and its CM-type is $\Sigma$. Fix such an abelian variety $A$ and assume that the index of the order $\textup{End}_K(A)$ inside the maximal order $\oo_K$ is prime to $p$. \emph{We will assume that $A$ is principally polarized and that $K$ contains its reflex field of the CM pair $(K,\Sigma)$.} Let $\mathcal{A}=\mathcal{A}_{/\frak{o}_K}$ denote the N\'eron model of $A$.

\subsubsection{Gr\"ossencharacters of CM abelian varieties}
\label{subsubsec:grossenkaraktere}
The $p$-adic Tate-module  $T_p(A)$ of $A$ is a free $\ZZ_p$-module of rank $2g$ on which $G_K$ acts continuously. As explained in the Remark on page 502 of \cite{serretate}, $T_p(A)$ is free of rank one over $\oo_K\otimes\ZZ_p=\prod_{\frak{q}} \frak{o}_\frak{q}$, where the product is over the primes of $K$ that lie above $p$ and $\frak{o}_\frak{q}$ stands for the valuation ring of $K_\frak{q}$. We thus have a decomposition $T_p(A)=\bigoplus_{\frak{q}}T_\frak{q}(A)$, where each $T_\frak{q}(A)=\varprojlim A[\frak{q}^n]$ is a free $\frak{o}_\frak{q}$-module of rank one. The $G_K$-action on $T_p(A)$ gives rise to characters $\psi_\frak{q}: G_K \ra \oo_\frak{q}^\times.$
It follows from \cite[\S 2]{ribetcompositio} that each $\psi_\frak{q}$ is surjective for $p$ large enough; we fix until the end a prime $p$ satisfying this condition. We thence obtain a decomposition
$$T_p(A)\otimes_{\ZZ_p}\overline{\QQ}_p=\bigoplus_{\frak{q}\mid p}\bigoplus_{\sigma: K_\frak{q} \hookrightarrow \overline{\QQ}_p} V_{\frak{q}}^\sigma$$
where $V_{\frak{q}}^\sigma$ is the one-dimensional $\overline{\QQ}_p$-vector space on which $G_K$ acts via the character
$\psi_\frak{q}^\sigma$, which is the compositum
$$G_K\stackrel{\psi_\frak{q}}{\lra} \frak{o}_\frak{q}^\times \stackrel{\sigma}{\lra} \overline{\QQ}_p^\times\,.$$
Fix embeddings $j_\infty:\overline{\QQ}\hookrightarrow \mathbb{C}$ and $j_p:\overline{\QQ}\hookrightarrow \mathbb{C}_p$ extending $\iota_p$. We write $\boldsymbol{\mathscr{S}}=\Sigma\cup\Sigma^c$ for the set of all embeddings of $K$ into $\overline{\QQ}$.
Theory of CM associates a Gr\"ossencharacter character 
$${\bbpsi}: \mathbb{A}_K/K^\times\lra K^{\times},$$
to $A$, which in turn induces Hecke characters
$$\psi_\tau\,:  \mathbb{A}_K/K^\times \stackrel{{\bbpsi}}{\lra} K^\times\stackrel{j_\infty\circ\tau}{\lra} \mathbb{C}^\times$$
as well as gives rise to its $p$-adic avatars
$$\psi_\tau^{(p)}\,:\mathbb{A}_K/K^\times \stackrel{{\bbpsi}}{\lra}K^\times \stackrel{j_p\circ\tau}{\lra} \mathbb{C}_p^\times.$$
Furthermore, the two sets $\{\textup{rec}\circ \psi_\tau^{(p)}\}_{\tau\in \frak{J}}$ and $\{\psi_\frak{q}^\sigma\}_{\frak{q},\sigma}$ of $p$-adic Hecke characters may be identified, where $\textup{rec}:\mathbb{A}_K/K^\times\ra G_K$ is the reciprocity map. Since we assumed that the field $K$ contains the reflex field of $(K,\Sigma)$, the Hasse-Weil $L$-function $L(A/K,s)$ of $A$ then factors into a product of Hecke $L$-series\footnote{Until the end of this article we shall write $L(\star,s)$ for the motivic $L$-functions and $L(s,\star)$ for the automorphic $L$-functions (so that the latter is centered at $s=1/2$).}
$$L(A/K,s+1/2)=\prod_{\tau\in\mathscr{S}} L(s,\psi_\tau^u) \in K\otimes \mathbb{C}\cong\mathbb{C}^{\mathscr{S}}$$
where $\psi^u$ is the unitarization of $\psi$. Fix $\varepsilon \in \Sigma$ and identify $K$ with $K^{\varepsilon}$. This choice (together with the chosen embeddings $j_\infty$ and $j_p$) in turn fixes a prime $\frak{p} \in \Sigma_p$ and $\sigma: K_{\frak{p}} \hookrightarrow \overline{\QQ}_p$ in a way that $\textup{rec}\circ\psi_\varepsilon^{(p)}=\psi_\frak{p}^\sigma$. We set $L=\sigma(K_\frak{p})$ and $\frak{o}:=\sigma(\oo_{{\frak{p}}})$ and define the $p$-adic Hecke character $\psi:=\psi_{\frak{p}}^\sigma: G_K \twoheadrightarrow \frak{o}^\times\,.$
 We set $\psi^*=\chi_{\textup{cyc}}\psi^{-1}$ and $T:=\frak{o}(\psi^*)$; note that we have $T^*\cong A[\varpi^{\infty}]$ (where $\varpi \in \mm_L$ is a uniformizer).
In this situation, our non-anomaly hypothesis {\textbf{\textup{(H.nA)}}} translates into the requirement that
\be\label{eqn:hna}   A(K_\frak{q})[\varpi]=0 \hbox{  for every prime } \frak{q}  \hbox{ of } K \hbox{ above } p\ee
which we assume throughout this subsection.

\emph{We assume in addition that $A$ arises as the base-change of an abelian variety (which we still denote by $A$, by slight abuse) that has real multiplication by $K_+$.} This condition implies that $\bbpsi$ is self-dual, in the sense that for each $\tau\in\mathscr{S}$ there is a functional equation with sign $\epsilon(\bbpsi/K)=\pm1$ (that does not depend on the choice of $\tau$) relating the value $L(s,\psi_\tau)$ to the value $L(1-s,\psi_\tau)$. 

We let $\frak{P}=\frak{p}\frak{p}^c$ denote the prime of $K_+$ below the prime $\wp$ we have fixed above and let $\varpi:=\varpi\varpi^c \in K_{+,\frak{P}}$ be a uniformizer. We write $T_{\frak{P}}(A):=\varprojlim A[\varpi^n]$ for the $\frak{P}$-adic Tate module of $A$ and set $\TT_{\frak{P}}(A):=T_{\frak{P}}(A)\otimes\LL$. By the theory of complex multiplication and our assumption on $A$, we have $T_{\frak{P}}(A)=\textup{Ind}_{K/K_{+}}T$. For every prime $\frak{Q}$ of $K_+$ above $p$, we have a $p$-ordinary filtration $F^+_{\frak{Q}}T_\frak{P}(A)\subset T_\frak{P}(A)$ (that also gives rise to the filtration $F^+_{\frak{Q}}\TT_\frak{P}(A)\subset \TT_\frak{P}(A)$) when $T_\frak{P}(A)$ is restricted to $G_{K_{\frak{Q}}}$. 

We  finally let $\phi$ denote the Hilbert modular CM form of weight two associated to $\phi_\epsilon$, $L(\phi,s)$ the associated Hecke $L$-function and $r_{\textup{an}}(\phi)$ the order of vanishing at the central critical point $s=1$. In what follows, we write $\LL_{\frak{o}}$ in place of $\LL\otimes\frak{o}$. 
\subsubsection{Perrin-Riou-Coleman maps and Selmer structures}
We introduce the Selmer structures we shall apply our theory in Section~\ref{subsec:KSway} with. 
\begin{define}
We define the Greenberg-submodule
$$H^1_{\textup{Gr}}(K_p,\TT):=\bigoplus_{\frak{q}\in \Sigma_p^c} H^1(K_{\frak{q}},\TT) \subset H^1(K_p,\TT)\,.$$
The semi-local Shapiro's lemma induces an isomorphism
\be
\label{semilocalshapiro}
\frak{sh}:\,H^1(K_{+,p},\TT_\frak{P}(A)) \stackrel{\sim}{\lra} H^1(K_p,\TT)
\ee
under which 
$$H^1_{\textup{Gr}}(K_{+,p},\TT_{\frak{P}}(A)):=\bigoplus_{\wp\mid p} \textup{im}\left(H^1(K_{+,\wp},F^+_{\wp}\TT_\frak{P}(A))\ra H^1(K_{+,\wp},\TT_\frak{P}(A) \right)$$
maps isomorphically onto $H^1_{\textup{Gr}}(K_p,\TT)$.
\end{define}

For each prime $\wp$ of $K_+$ above $p$, we let $D_\wp(T_{\frak{P}}(A))$ denote the Dieudonn\'e module of $T_{\frak{P}}(A)$ considered as a $G_{K_{+,\wp}}$-representation. As explained in \cite{kbbleiPLMS}, we may (and we will) think of this as an $\frak{o}$-module of rank $2f_{\wp}$ (where $f_{\wp}=[K_{+,\wp}:\QQ_p]$). We also let
$$\exp^*:\, H^1_{/f}(K_{+,\wp},T_{\frak{P}}(A))\lra \textup{Fil}^0\,D_\wp(T_{\frak{P}}(A))$$
the Bloch-Kato dual exponential map,
$$\log_{A,\wp}:  H^1(K_{+,\wp},T_{\frak{P}}(A))\lra D_\wp(T_{\frak{P}}(A))/\textup{Fil}^0\,D_\wp(T_{\frak{P}}(A))$$
the inverse of the Bloch-Kato exponential map and 
\be\label{eqn:naturalhodgepairing}\left[\,,\,\right]\,: D_\wp(T_{\frak{P}}(A)) \times D_\wp(T_{\frak{P}}(A)) \lra \frak{o}
\ee
the canonical perfect pairing induced from the Weil-pairing (thanks to which we have an identification $D_\wp(T_{\frak{P}}(A))^*:=D_{\wp}(T_{\frak{P}}(A)^\mathcal{D}(1))\cong D_\wp(T_{\frak{P}}(A))$), where for an $\frak{o}$-module $M$, we write $M^{\mathcal{D}}$ for its $\frak{o}$-linear dual. We let $D_\wp(T_{\frak{P}}(A))_{[-1]}=D(F^+T_\PPP(A))$ denote the subspace of $D_\wp(T_{\frak{P}}(A))$ on which the crystalline Frobenius $\varphi$ acts with slope $-1$ and let 
$$\omega^*_{\wp}=\{\omega_{i,\wp}^*\}_{i=1}^{f_\wp} \subset \textup{Fil}^0D_\wp(T_{\frak{P}}(A))^*\otimes_{\ZZ_p}\QQ_p$$ 
denote a fixed basis corresponding (under the comparison isomorphism) to the N\'eron differential\footnote{This is a top degree invariant form on a N\'eron model of $A$ and as such, does only determine the top-degree exterior product $\wedge \omega^*:=\omega_{1,\wp}^*\wedge \cdots\wedge \omega_{f_\wp,\wp}^*$ uniquely.} on $A$. Since we have 
$$D_\wp(T_{\frak{P}}(A))_{[-1]}\, \cap\, \textup{Fil}^0D_\wp(T_{\frak{P}}(A))=0,$$ we may choose a basis $\omega_\wp=\{\omega_{j,\wp}\}_{j=1}^{f_\wp}$ of $D_\wp(T_{\frak{P}}(A))_{[-1]}$ such that $[\omega_{i,\wp}^*,\omega_{j,\wp}]=\delta_{i,j}$, where $\delta_{i,j}$ is the Kronecker-delta. Let $\omega$ denote the collection $\{\omega_\wp\}_{\wp\,\mid\, p}$ of these distinguished bases. The pairing (\ref{eqn:naturalhodgepairing}) induces a commutative diagram
$$
\xymatrix{
D_\wp(T_{\frak{P}}(A))_{[-1]}\ar[d]& \times &D_\wp(T_{\frak{P}}(A)/F^+T_\PPP(A)) \ar[r]^(.77){\left[\,,\,\right]}& \frak{o}\ar@{=}[d]\\
D_\wp(T_{\frak{P}}(A))& \times &D_\wp(T_{\frak{P}}(A)) \ar[r]^(.65){\left[\,,\,\right]}\ar[u]& \frak{o}
}$$
which in turn allows us to associate each $\omega_{j,\wp}\in D_\wp(T_{\frak{P}}(A))_{[-1]}$ a map (which we still denote by the same symbol) \,$\omega_{j,\wp}: D_\wp(T_{\frak{P}}(A)) \ra \frak{o}$ that factors through the quotient $D_\wp(T_{\frak{P}}(A)/F^+T_{\frak{P}}(A))$. With a slight abuse, we also write $\omega$ for the map
$$\omega=\oplus_{j=1}^{f_\wp} \omega_{j,\wp}:\, D_\wp(T_{\frak{P}}(A)) \lra \frak{o}^{\oplus f_\wp}$$
that also factors through $D_\wp(T_{\frak{P}}(A)/F^+T_{\frak{P}}(A))$.
\begin{thm}[Perrin-Riou, \cite{PRmap}]
\label{thm:PRmapintegralproperties}
There exists a $\LL_\frak{o}$-equivariant map 
$$\frak{L}^{(\wp)}_\omega\,:\,H^1(K_{+,\wp},\TT_{\frak{P}}(A)) \lra \LL_\frak{o}^{\oplus f_{\wp}}$$
which interpolates (in a sense we will not make precise here) the dual exponential maps along the cyclotomic Iwasawa tower.  Furthermore, the kernel of the map $\frak{L}^{(\wp)}_\omega$ is precisely the Greenberg submodule $H^1_\Gr(K_{+,\wp},\TT_{\frak{P}}(A))$ and it is pseudo-surjective.
\end{thm}
\begin{proof}
To simplify notation, we shall write $V$ in place of $\QQ_p\otimes T_\frak{P}(A)$\,;\, $F^+\,V$ in place of  $\QQ_p\otimes F^+_\frak{P}\,T_\frak{P}(A)$\,;\, $D(V)$ in place of $D_\wp(T_{\frak{P}}(A))\otimes\QQ_p$ and finally $\Phi$ in place of $K_{+,\wp}$. Only in this proof, let us also write $\LL$ for $\LL_\frak{o}$. We recall once again that $\Phi/\QQ_p$ is unramified. Let $\mathcal{H}(\Gamma)\subset L[[\Gamma]]$ denote Perrin-Riou's ring of distributions. Then Perrin-Riou's original construction yields a map  
$$\frak{L}_V:\, H^1(\Phi,V\otimes\LL)\lra \mathcal{H}(\Gamma)\otimes D(V)\,.$$
Let $\frak{L}_{V,\omega}: H^1(\Phi,V\otimes\LL)\ra \mathcal{H}(\Gamma)^{\oplus f_\wp}$ denote the map $\omega\circ \frak{L}_V$, where $\omega: D(V)\ra L$ is the map we have introduced above (which in fact factors through $D(V/F^+V)$). It follows from the discussion in \cite[Section 3]{jaycycloIwasawa} that the following diagram commutes up to units in $\mathcal{H}(\Gamma)$:
$$\xymatrix{H^1(\Phi,V\otimes\LL)\ar[rr]^(.6){\frak{L}_{V,\omega}}\ar@{->>}[d]&&\mathcal{H}(\Gamma)^{\oplus f_\wp}\ar[d]^{\bf{id}}\\
H^1(\Phi,V/F^+V\otimes\LL)\ar[rr]^(.6){\frak{L}_{V/F^+V,\omega}}&&\mathcal{H}(\Gamma)^{\oplus f_\wp}
}$$
Here $\frak{L}_{V/F^+V,\omega}:=\omega\,\circ\,\frak{L}_{V/F^+V}$ and $\frak{L}_{V/F^+V}$ is the Perrin-Riou map for $V/F^+V$. We remark that we may take the right vertical arrow to be the identity map since the Hodge-Tate weights of $V$ are $0$ and $1$, and since 
$$H^1(\Phi,V/F^+V\otimes\LL)_{\textup{tor}}=H^2(\Phi,V/F^+V\otimes\LL)=0$$ 
thanks to our assumption (\ref{eqn:hna}). A suitable $p$-power multiple of the compositum of the arrows 
$$H^1(\Phi,V\otimes\LL)\ra H^1(\Phi,V/F^+V\otimes\LL) \stackrel{\frak{L}_{V/F^+V,\omega}}{\lra}\mathcal{H}(\Gamma)^{\oplus f_\wp}$$
is the map we denote by $\frak{L}_\omega^{(\wp)}$ above; the fact that this compositum takes values in $\mathcal{H}_0(\Gamma)^{\oplus f_\wp}=\LL[1/p]$ follows from the fact that $D(V/F^+V)$ has slope $0$. It also follows from \cite[Section 3]{jaycycloIwasawa} that the lower horizontal map in the diagram above is injective and therefore the portion of Theorem~\ref{thm:PRmapintegralproperties} concerning the kernel of $\frak{L}_{\omega}^{(\wp)}$ is proved. Finally, Theorem 3.4 of loc. cit shows that the map
$$H^1(\Phi,V/F^+V\otimes\LL)\otimes_\LL\mathcal{H}(\Gamma) \stackrel{\frak{L}_{V/F^+V,\omega}}{\lra}\mathcal{H}(\Gamma)^{\oplus f_\wp}$$
is surjective and the portion concerning the image of $\frak{L}_{\omega}^{(\wp)}$ also follows.
\end{proof}

We shall write $\frak{L}^{(\wp,i)}_\omega$ for the map 
$$\frak{L}^{(\wp,i)}_\omega\,:\,H^1(K_{+,\wp},\TT_{\frak{P}}(A)) \lra \LL_\frak{o}^{\oplus (f_{\wp}-1)}$$
for the map obtained from $\frak{L}^{(\wp)}_\omega$ by omitting the summand in the target which corresponds to $\omega_{i,\wp}$. Finally, we set 
$$\frak{L}_\omega:=\oplus_{\wp\mid p}\,\frak{L}_\omega^{(\wp)}:=H^1(K_{+,p},\TT_\frak{P}(A))\lra \LL_{\frak{o}}^g,$$
where $g=[K_{+}:\QQ]$ and as well last the dimension of the abelian variety $g$. We note that $\ker\left(\frak{L}_\omega\right)=H^1_\Gr(K_{+,p},\TT_\PPP(A))$.
 
\begin{define}
\label{def:plusselmergroupfortheabvar}
We fix a prime $\frak{Q}$ of $K_+$ above $p$, as well as an integer $1\leq i\leq f_{\frak{Q}}$. We define the map 
$$\frak{L}_{(\frak{Q},i)}=\bigoplus_{\frak{Q}\neq\wp\,\mid\, p}\frak{L}^{(\wp)}_\omega \oplus \frak{L}^{(\frak{Q},i)}_\omega\,:\,H^1(K_{+,p},\TT_\frak{P}(A))\lra \LL_{\frak{o}}^{g-1}\,.$$
We set $H^1_{+}(K_{+,p},\TT_\frak{P}(A)):=\ker(\frak{L}_{(\frak{Q},i)})$ and define $H^1_{+}(K_{p},\TT)$ as the isomorphic image of $H^1_{+}(K_{+,p},\TT_\frak{P}(A))$ under Shapiro's morphism $\frak{sh}$. We define the Selmer structure $\FF_+$ on $\TT$ by the requiring that 
$$H^1_{\FF_+}(K_p,\TT)=H^1_{+}(K_{p},\TT) \hbox{ and } H^1_{\FF_+}(K_\lambda,\TT)=H^1(K_{\lambda},\TT) \hbox{ for } \lambda \nmid p\,. $$
\end{define}
In this situation, Theorem~\ref{thm:mainkbb} applies since we assumed (\ref{eqn:hna}). Furthermore, \emph{if we assume the truth of Perrin-Riou-Stark Conjecture} proposed in \cite{kbbleiPLMS} (which is a slightly strong form of Rubin-Stark conjectures) \emph{as well as Leopoldt's conjecture for all subextensions of $K(A[\varpi])/K$}, we may obtain a natural generator of this module using the main results of \cite{kbbCMabvar,kbbleiPLMS}; this is what we explain in the next paragraph.

\subsubsection{The Coleman-Rubin-Stark element}
\label{subsubsec:CRSelement}
Let $K_\cyc=K\QQ_\infty$ denote the cyclotomic $\ZZ_p$-extension of $K$. Since we assumed that $p$ is unramified in $K/\QQ$, we may canonically identify $\Gamma$ with $\Gal(K_\cyc/K)$. Let $K_\infty$ denote the maximal $\ZZ_p$-power extension of $K$and  $\Gamma_K=\Gal(K_\infty/K)$ its Galois group over $K$. 

Let $\omega_\psi$ denote the character of $G_K$ giving its action on $A[\varpi_E]$; it is the unique character of $G_K$ which has the properties that the character $\langle\psi\rangle:=\psi\omega_\psi^{-1}$ factors through $\Gamma_K$ and that  it is trivial on $\Gamma_K$. Let $T_{\omega_\psi}=\frak{o}(1)\otimes\omega_\psi^{-1}$ and define $H^1_\infty(K,T_{\omega_\psi}):=\varprojlim H^1(F,T_{\omega_\psi})$, where the projective limit is over all finite sub-extensions of $K_\infty/K$. Assume the truth of the \emph{Perrin-Riou-Stark Conjecture 4.14}  in \cite{kbbleiPLMS} (with the Dirichlet character $\omega_\psi$) and let\footnote{We invite the interested reader to consult \cite[Remark 4.13]{kbbleiPLMS} and the discussion that precedes this remark for the desired integrality properties of the Rubin-Stark elements.} $\frak{S}_\infty^{\omega_\psi} \in \wedge^{g} H^1_\infty(K,T_{\omega_\psi})$ denote the element whose existence is predicted by the said conjecture. As in Definition 4.16 of loc. cit., we may twist this element to obtain the twisted Perrin-Riou-Stark elements $\frak{S}_\infty \in \wedge^g H^1_\infty(K,T)$ as well as their projections 
$$\frak{S}_\cyc=\frak{S}_\cyc^{(1)}\wedge\cdots\wedge \frak{S}_\cyc^{(g)} \in \wedge^g H^1(K,\TT)\cong \wedge^g H^1(K_+,\TT_\PPP(A))$$
(where we denote the image of $\frak{S}_\cyc$ under the isomorphism above still by the same symbol). We finally set 
$$\mathbb{RS}^\psi_{/f}:=\textup{res}_{/f}^{\otimes g} (\frak{S}_\cyc) \in \wedge^g H^1_{/f}(K_p,\TT)\cong \wedge^g H^1_{/f}(K_p,\TT_\frak{P}(A))\,.$$
\begin{thm}
\label{thm:mainRS}
There exists a generator $\bbkappa^{\textup{CRS}}$ (the Coleman-adapted Rubin-Stark Kolyvagin system) of the free $\LL$-module $\KS(\FF_+,\TT,\PP)$ whose initial term $\kappa_1^{\textup{CRS}} \in H^1_{\FF_+}(K_p,\TT)$ has the following property:
$$\frak{L}_{(\frak{Q},i)}\left(\res_{/f}\left({\kappa_1^{\textup{CRS}} }\right)\right)=\frak{L}_\omega^{\otimes g}\left(\mathbb{RS}_{/f}^\psi\right)$$
\end{thm}
\begin{proof}
This is precisely the content of Theorem A.11 and  Proposition A.8 of \cite{kbbleiPLMS} (except that for our purposes, it is sufficient to restrict our attention to the case when $\LL$ in loc. cit. has also Krull dimension $2$). In order to apply these results, we simply choose $\frak{L}_\omega$ in place of $\Psi$ and the summand in the target of $\frak{L}_\omega^{(\frak{Q})}$ that corresponds to $\omega_{i,\frak{Q}}$ in place of $L(1)$ in loc. cit. Note in this case that the Selmer structure $\FF_+$ above corresponds to the Selmer structure denoted by $\FF_L$ in loc. cit.
\end{proof}
\begin{define}
\label{def:ColemanRS}
We let $\crs_\infty \in H^1_{\FF_+}(K_+,\TT_\PPP(A))$ denote the element that corresponds to $\kappa_1^{\textup{CRS}} $ under the isomorphism $\frak{sh}$ and let $\crs \in H^1_{\FF_+}(K_+,T_\PPP(A))$ (that we call the \emph{Coleman-Rubin-Stark class}) its obvious projection. 
\end{define}
The following follows as an immediate consequence of Theorem~\ref{thm:PRwithKS} and Theorem~\ref{thm:mainRS}:
\begin{cor}
\label{cor:PRconjforRSelements}
The map $\res_{f/-}: H^1_f(K_+,T_{\frak{P}}(A))\ra H^1_{f/-}(K_p,T_\PPP(A))$ is injective if and only if the Coleman-Rubin-Stark class $\crs$ is non-trivial.
\end{cor}

\subsubsection{Katz' $p$-adic $L$-function an explicit reciprocity conjecture for Rubin-Stark elements} We recall here the definition the $p$-adic $L$-function of Katz and Hida-Tilouine and propose an extension of the Coates-Wiles reciprocity law for elliptic units to a reciprocity law concerning the Perrin-Riou-Stark elements.
\begin{define}
A pair $(m_0,d)$ (where $m_0 \in \ZZ$ and $d=\sum_{\sigma\in \Sigma} d_\sigma\sigma \in \ZZ^\Sigma$)  is called \emph{$\Sigma$-critical} if either $m_0>0$ and $d_\sigma\geq 0$ or else $m_0\leq 1$ and $d_\sigma\geq 1-m_0$ for every $\sigma \in \Sigma$. 

Likewise, a Gr\"ossencharacter $\bblambda$ is called \emph{$\Sigma$-critical} if its infinity type equals the expression $\sum_{\sigma \in \Sigma} (m_0+d_\sigma)\sigma-d_{\sigma}\sigma^c \in \ZZ^{\mathscr{S}}$ for some $\Sigma$-critical pair $(m_0,d)$.
\end{define}
Let $\mathscr{O}$ denote the $p$-adic completion of the maximal unramified extension of $\QQ_p$. Let $\Gamma_K$ denote the Galois group of the maximal $\ZZ_p$-power extension of $K$. By class field theory, note that we have $\Gamma_K\cong \ZZ_p^{1+g+\delta}$, where $\delta$ is Leopoldt's defect (which equals zero whenever we assume Leopoldt's conjecture). 

The following statement was proved by Hida and Tilouine in \cite{hidatilouine}, extending a previous construction due to Katz. In our discussion below, we will mostly stick to the exposition in \cite{hsiehanticyclopadicL} and we shall rely on Hsieh's notation (except perhaps minor alterations). Let $\frak{f} \subset \frak{o}_K$ denote the conductor of $\bbpsi$. Write $\frak{f}=\frak{f}^+\frak{f}^-$ where $\frak{f}^+$ (resp., $\frak{f}^-$) is a product of primes that split (resp., that remain inert or ramify) over $K_+$. The following is Proposition~4.9 in \cite{hsiehanticyclopadicL}.
\begin{thm}[Katz, Hida-Tilouine]
There exists an element $\mathscr{L}^{\Sigma}_{p}\in \mathscr{O}[[\Gamma_K]]$ $($Katz' $p$-adic $L$-function$)$ that is uniquely determined by the following interpolation property on the $p$-adic avatars $\bblambda^{(p)}=(\lambda_\tau^{(p)})$ of the $\Sigma$-critical characters $\bblambda=(\lambda_\tau)$ of infinity type $m_0\Sigma+(1-c)d \in \ZZ^\mathscr{S}$  and of conductor dividing $\frak{f}p^\infty$:
$$\frac{\mathscr{L}^{\Sigma}_{p}(\bblambda^{(p)})}{\Omega_p^{m_0\Sigma+2d}}=t\cdot\frac{\pi^d\,\Gamma_{\Sigma}(m_0\Sigma+d)}{\sqrt{|D_{K_+}|_{\mathbb{R}}}\,\textup{im}(\delta)^d}\cdot\mathscr{E}_p(\bblambda)\mathscr{E}_{\frak{f}^+}(\bblambda)\cdot\prod_{\frak{\frak{q}\mid \frak{f}}p}(1-\bblambda(\frak{q}))\cdot \frac{L(m_0/2,\bblambda^u)}{\Omega_\infty^{m_0\Sigma+2d}}
$$
where the equality takes place in $\iota_p(\overline{\QQ})^\mathscr{S}$ and where
\begin{itemize}
\item $\Omega_p=(\Omega_p(\sigma))_{\sigma} \in \mathscr{O}^\Sigma$ and $\Omega_\infty=(\Omega_\infty(\sigma))_{\sigma} \in \mathbb{C}^\Sigma$ are the periods which are attached to a N\'eron differential $\omega$ on the abelian scheme $\mathcal{A}$, as in \cite[Chapter II]{katz78};
\item $t$ is a certain fixed power $2$ (which can be made explicit);
\item $\delta \in K_+$ is the element chosen as in \cite[0.9a-b]{hidatilouine} and \cite[Section 3.1]{hsiehanticyclopadicL};
\item $\mathscr{E}_p(\bblambda)$ and $\mathscr{E}_{\frak{f}^+}(\bblambda)$ are products of modified Euler factors defined in \cite[(4.16)]{hsiehanticyclopadicL} and denoted by $Eul_p$, $Eul_{\frak{f}^+}$ in loc. cit.,
\item $\bblambda^u:=\bblambda/|\bblambda|_{\mathbb{A}}$ is the unitarization of $\bblambda$ $($where this terminology is borrowed from Hida and Tilouine \cite[pp. 231-232]{hidatilouine}$)$.
\end{itemize}
\end{thm}
In particular, when the Gr\"ossencharacter $\bblambda$ has infinity type $\Sigma$ (so that $m_0=1$ and $d=0$) and conductor $\frak{f}$, the interpolation formula simplifies to
$$\frac{\mathscr{L}^{\Sigma}_{p}(\bblambda^{(p)})}{\Omega_p^{\Sigma}}=t\cdot\mathscr{E}_p(\bblambda)\mathscr{E}_{\frak{f}^+}(\bblambda)\cdot\prod_{\wp |p}(1-\bblambda(\wp))
\cdot\frac{L(1/2,\bblambda^u)}{\sqrt{|D_{K_+}|_{\mathbb{R}}}\,\Omega_\infty^{\Sigma}}$$
We call the pullback $\mathscr{L}^{\Sigma}_{p,\bbpsi}$ of $\mathscr{L}^{\Sigma}_{p}$ along the character $\bbpsi$ (where $\bbpsi$ is the Gr\"ossencharacter associated to $A$) the \emph{$\bbpsi$-branch of $\mathscr{L}^{\Sigma}_{p}$}. In more concrete terms we have
$$\mathscr{L}^{\Sigma}_{p,\bbpsi}(\chi):=\mathscr{L}^{\Sigma}_{p}(\chi\bbpsi^{(p)})\,$$
for every character $\chi$ of finite order.

Recall the $p$-adic Hecke character $\psi$ we have fixed above, which is associated to our choice of $\epsilon\in \Sigma$ (and fixed embeddings $j_\infty$ and $j_p$). We will fix the modulus $\frak{f}$ to be chosen as the conductor of $\psi$.
\begin{cor}
\label{cor:padicLfunctionforA}
For every character $\chi$ of $\Gamma_K$ of finite order we have,
$$\frac{\mathscr{L}^{\Sigma}_{p}(\chi\psi)}{\Omega_p(\epsilon)}=\mathscr{E}(\chi\psi)\prod_{\wp|p}(1-\chi_\epsilon\psi_\epsilon(\wp))
\cdot\frac{L(1/2,\chi_\epsilon\psi_\epsilon^u)}{\Omega_\infty(\epsilon)}$$
where $\mathscr{E}(\chi\psi):=t\cdot\mathscr{E}_p(\chi\psi)\cdot\mathscr{E}_{\frak{f}^+}(\chi\psi)$ is a product of Euler factors up to an explicit power of $2$.
\end{cor}
\begin{rem}
\label{rem:comparedisegnitohsiehepsilontau}
Note that the Euler like factors $1-\chi_\epsilon\psi_\epsilon(\wp)$ are equal to $1$ so long as $\chi$ is not the trivial character.
\end{rem}
\begin{define}
We let $\mathscr{L}_\textup{cyc}^\Sigma \in \mathscr{O}[[\Gamma]]$ denote the measure obtained by restricting the $\bbpsi$-branch $\mathscr{L}^{\Sigma}_{p,\bbpsi}$ of the Katz $p$-adic $L$-function to the cyclotomic characters of finite order.
\end{define}
\begin{define}
Let $n$ be a positive integer and $z=\{z_n\}\in H^1(K_+,\TT_\frak{P}(A))$ be an arbitrary element and let $\chi$ be a primitive character of $\Gamma_n$. Fix a prime $\wp$ of $K_+$ lying above $p$. Set $K_n^+:=K_+\QQ_n$ and denote by $\wp_n$ the unique prime of $K_n^{+}$ above $\wp$. For every $1\leq j\leq f_{\wp}$\,, we define the \emph{Perrin-Riou symbol} $[z,\wp,j,\chi]$ by setting
$$[z,\wp,j,\chi]:=\frac{1}{\tau(\chi)}\left[\sum_{\gamma\in\Gamma_n}\chi(\gamma)\exp^*_n(\res_{\wp_n}\left(z_n\right)^\gamma)\,,\, \varphi^{-n}(\omega_{j,\wp})\right]$$
where $\tau(\chi)$ is the Gauss sum and 
$$\exp^*_n:\, H^1(K_{n,\wp_n},T_{\PPP}(A))\lra \QQ_{p,n}\otimes_{\ZZ_p}\textup{Fil}^0D_\wp(T_\PPP(A))$$
is the dual exponential map. On fixing an ordering of primes of $K_+$ above $p$, we let $\mathbb{R}_{\chi}(z,\omega,\chi)$ denote the $1\times g$ matrix given by 
$$\mathbb{R}_{\chi}\left(z,\omega,\chi\right)=\left([z,\wp,j,\chi]\right)_{\substack{\wp\,\mid\, p\\1\leq j\leq f_\wp}}\,.$$
For an element $\mathbf{z}=z_1\wedge\cdots\wedge z_g \in \wedge^g H^1(K_+,\TT_\frak{P}(A))$, we further define the $g\times g$ matrix $\mathbb{M}(\mathbf{z},\omega,\chi)$ by setting
$$\mathbb{M}(\mathbf{z},\omega,\chi)=\left(\mathbb{R}_{\chi}\left(z_i,\omega,\chi\right)\right)_{i=1}^g
$$
We remark that this definition actually depends only on the element $\wedge \omega^*$ (that corresponds to a N\'eron differential on $A$) and not the choice of a basis that represents this differential.
\end{define}
We propose the following explicit reciprocity conjecture for the twisted Perrin-Riou-Stark element 
$$\frak{S}_\cyc=\frak{S}_\cyc^{(1)}\wedge\cdots\wedge \frak{S}_\cyc^{(g)} \in \wedge^g H^1(K_+,\TT_\frak{P}(A))$$
as a natural extension of Coates-Wiles explicit reciprocity law for elliptic units. 
\begin{conj}
\label{conj:explicitreciprocityforRS}
There exists a choice of a N\'eron differential on $A$ so that we have
$$\det\mathbb{M}(\frak{S}_\cyc,\omega,\chi)=
\mathscr{E}(\chi\psi)\cdot\frac{L(1/2,\chi_\epsilon\psi_\epsilon)}{\Omega_\infty(\epsilon)}\,$$
for every primitive character $\chi$ of $\Gamma_n$\,.
\end{conj}
Whenever we assume the truth of Conjecture~\ref{conj:explicitreciprocityforRS}, we shall implicitly assume also that we are working with a basis $\omega$ as comes attached to the appropriate choice of a N\'eron differential.
\begin{cor}
\label{cor:explicitreciprocityconj}
If the Explicit Reciprocity Conjecture~\ref{conj:explicitreciprocityforRS} holds true, we have
$${\displaystyle\frak{L}^{\otimes g}_\omega\left(\mathbb{RS}_{/f}^\psi\right)=\frac{\mathscr{L}_\textup{cyc}^\Sigma}{\Omega_p(\epsilon)}}\,.$$
\end{cor}
\begin{proof}
This follows at once making use of the displayed equation (5) in \cite{kbbleiintegralMC}.
\end{proof}
\subsubsection{Perrin-Riou's conjecture for CM abelian varieties}
\label{subsubsec:PRconj}
We may finally turn our attention to Conjecture~\ref{conj:analyticrankvslocalizationmapinjective} in this set up. \emph{We will assume until the end of Section~\ref{sec:CMabvarcolemanRS} that there exists a degree one prime of $K_+$ above $p$} (mostly for brevity; we expect that one should be able to get around of this assumption with more work) and we choose the prime $\frak{Q}$ above that we work with as this prime of degree one. Note in this case that $i=f_{\frak{Q}}=1$ and also that 
$$\ker\left(\frak{L}^{(\frak{Q},1)}\right)=H^1(K_{+,\frak{Q}},\TT_\PPP(A))\oplus\bigoplus_{\frak{Q}\neq \frak{p}\mid p}H^1_{\FF_\Gr}(K_{+,\frak{p}},\TT_\frak{P}(A)).$$
Also in this situation, note that the set $\omega_{\frak{Q}}$ is a singleton, and by slight abuse, we denote its only element also by $\omega_{\frak{Q}}$.

\begin{define}
\label{def:heightparing}
We let
$$\langle\,,\,\rangle: H^1_f(K_+,T_\frak{P}(A))\otimes H^1_f(K_+,T_\frak{P}(A))\lra L$$
denote the $p$-adic height pairing of Perrin-Riou~\cite[Section 2.3]{PRheights} associated to the canonical unit root Hodge splitting, the cyclotomic character and Iwasawa's branch of the $p$-adic logarithm. 
\end{define}
The following is the version of Perrin-Riou conjecture in this set up:
\begin{thm}
\label{thm:PRconjforRSforreal}
If the Explicit Reciprocity Conjecture~\ref{conj:explicitreciprocityforRS} holds true and if either $r_{\textup{an}}(\phi)=0$ or else $r_{\textup{an}}(\phi)=1$ and the $p$-adic height pairing $\langle\,,\,\rangle$ is non-zero, then the Coleman-Rubin-Stark class $\crs$ is non-trivial.
\end{thm}
\begin{rem}
\label{rem:ifonlyKolyvaginLogachev}
If one could extend the main theorem of Kolyvagin and Logachev \cite{kolyvaginlogachev} to cover CM abelian varieties, one may prove Theorem~\ref{thm:PRconjforRSforreal}  without assuming either the explicit reciprocity conjecture or the non-triviality of the $p$-adic height pairing. 
\end{rem}
\begin{proof}[Proof of Theorem~\ref{thm:PRconjforRSforreal}]
When $r_{\textup{an}}(\phi)=0$, it follows from the interpolation formula for the $p$-adic $L$-function that $\mathds{1}(\mathscr{L}_\textup{cyc}^\Sigma)\neq 0$. The assertion in this case follows on combining Theorem~\ref{thm:mainRS} and Corollary~\ref{cor:explicitreciprocityconj}.  

Suppose now that $r_{\textup{an}}(\phi)=1$ and the $p$-adic height pairing $\langle\,,\,\rangle$ is non-zero. On choosing an auxiliary totally imaginary extension $F/K_+$ in suitable manner (in a way that the non-vanishing results of Friedberg and Hoffstein~\cite{friedberghoffstein} apply) and relying on the Gross-Zagier formula of Yuan-Zhang-Zhang in~\cite{YYZ} and its $p$-adic variant due to Disegni~\cite{disegniGZ}, we conclude that $\mathscr{L}_\textup{cyc}^{\Sigma} \not\equiv 0 \mod J^2$. Combined with the standard application of the Coleman-Rubin-Stark $\LL$-adic Kolyvagin system $\bbkappa^{\textup{CRS}}$ (and a control argument for Greenberg Selmer groups), it follows that the $\frak{o}$-module $H^1_f(K_{+},T_\PPP(A))$ is of rank one and that $\Sha(A/K_+)[\frak{P}^\infty]$ is finite. Thence the map $\res_{f/-}$ may be given explicitly via the commutative diagram
$$\xymatrix{H^1_{\FF_\Gr}(K_{+},T_\PPP(A))\ar[r]^{\res_{f/-}} &H^1_{\FF_\Gr}(K_{+,\frak{Q}},T_\PPP(A))\\
A(K_{+})\widehat{\otimes}\,\frak{o}\ar[r]^{\res_{f/-}} \ar[u]_{\cong}& A(K_{+,\frak{Q}})\widehat{\otimes}\,\frak{o}\ar[u]_{\cong}}$$
and it is evidently injective. The proof follows by Corollary~\ref{cor:PRconjforRSelements}.
\end{proof}

In what follows, we let $D_\frak{P}(A)$ be a shorthand for the Dieudonn\'e module of the $G_{K_{+,\frak{Q}}}$-representation $V_{\frak{P}}(A)$. Observe that $K_{+,\frak{Q}}=\QQ_p$ thanks to our choice of $\frak{Q}$.

\begin{thm}
\label{thm:ercimpliesonlyif}
The Explicit Reciprocity Conjecture~\ref{conj:explicitreciprocityforRS} implies the ``only if'' portion of Conjecture~\ref{conj:analyticrankvslocalizationmapinjective} whenever the $p$-adic height pairing of Definition~\ref{def:heightparing} is non-zero, $p$ is prime to $w_{2}(K_+):=H^0(K_+,\QQ/\ZZ(2))$ and $D_{\PPP}(A)^{\varphi=1}=0$.
\end{thm}
\begin{proof}
By Corollary~\ref{cor:PRconjforRSelements}, we may assume that the Coleman-Rubin-Stark element $\frak{S}$ is non-trivial and given that, we contend to prove that $r_{\textup{an}}(\phi)\leq 1$ in our set up. 

As above, let $\omega_{\mathcal{A}}$ denote a N\'eron differential on $\mathcal{A}$ and let 
$$\omega^*_\frak{Q} \in \textup{Fil}^0D_\frak{P}(A)^*\cong \textup{Fil}^0D_\frak{P}(A)$$
  denote the element that corresponds to $\omega_{A}$ under the comparison isomorphism. Let $D_{[-1]} \subset D_{\frak{P}}(A)$ denote the subspace of $D_{\frak{P}}(A)$ on which $\varphi$ acts with slope $-1$.  Then the space $D_{[-1]}$ is one-dimensional and  $D_{[-1]}\cap \textup{Fil}^0 D_{\frak{P}}(A)=\{0\}$. Since $\textup{Fil}^0 D_{\frak{P}}(A)$ is the exact orthogonal compliment of $\textup{Fil}^0 D_{\frak{P}}(A)^*$ under the pairing $[\,,\,]$ above, there exists a unique element $\omega_\frak{Q} \in D_{[-1]}$ with $[\omega_\frak{Q},\omega_\frak{Q}^*]=1$ (which in fact spans $D_{[-1]}$ as an $L$-vector space). We will denote the image of $\omega_\frak{Q}$ under the isomorphism $D_{[-1]}\stackrel{\sim}{\ra}D_{\frak{P}}(A)/\textup{Fil}^0 D_{\frak{P}}(A)$ by $\overline{\omega}_\frak{Q}$. As explained in \cite[Section 2.1]{kbbleiintegralMC}, we have
 \begin{align}\notag \mathds{1}\left(\frak{L}^{\frak{Q},1}(z_\infty)\right)&=\left[\exp^*(z_0),(1-p^{-1}\varphi^{-1})(1-\varphi)^{-1}\omega_\frak{Q}\right]\\
 \label{en:leadingtermformulaforPRlog}&=(1-1/\alpha)(1-\alpha/p)^{-1}\left[\exp^*(z_0),\omega_\frak{Q}\right]
 \end{align}
for every $z_\infty=\{z_n\}_{n\geq 0}\in H^1_{/f}(K_{+,\frak{Q}},\TT_{\frak{P}}(A))$, where $\alpha$ is the $p$-unit eigenvalue for $\varphi$ acting on $D_{\frak{P}}(A)$. Note that $\alpha\neq1$ by assumption.
\\
\emph{Case 1.} $\textup{res}_{/f}(\frak{C})\neq 0$. Under our running hypotheses, it follows from Corollary~\ref{cor:explicitreciprocityconj} and (\ref{en:leadingtermformulaforPRlog}) that $\mathds{1}(\mathscr{L}^\Sigma_{\textup{cyc}})\neq 0$. The interpolation property for this $p$-adic $L$-function now shows that $r_{\textup{an}}(\phi)=0$.
\\
\emph{Case 2.} $\textup{res}_{/f}(\frak{C})=0$. This means that $\frak{S} \in H^1_f(K_+,T_{\frak{P}}(A))$ and in turn, also that 
\begin{align*}
\frak{S}_{\infty} &\in \ker\left(H^1_{/f}(K_{+,\frak{Q}},\TT_{\frak{P}}(A))\ra H^1_{/f}(K_{+,\frak{Q}},T_{\frak{P}}(A))\right)\\
&=(\gamma-1)H^1_{/f}(K_{+,\frak{Q}},\TT_{\frak{P}}(A))
\end{align*}
It follows by Theorem~\ref{thm:mainRS} that $\mathds{1}(\mathscr{L}_\textup{cyc}^\Sigma)=0$ and the interpolation formula shows (as non of the Euler-like factors in its statement vanish) that $L(1/2,\psi_\varepsilon)=0\,.$ In this case, the discussion in \cite[\S11.3.14]{nek} applies and allows us to construct the derivative $\frak{dS}_\infty \in H^1_{/f}(K_{+,\frak{Q}},T_\PPP(A))$ of the class $\frak{S}_\infty$ (our element $\frak{dS}_\infty$ corresponds to the image of the element $(Dx_\Iw)_{\frak{Q}}$ in loc. cit. under the cyclotomic character).

It then follows from Nekov\'a\v{r}'s Rubin-style formula \cite[Proposition 11.3.15]{nek} (also, his $p$-adic height compares via \S11.3 to those introduced by Perrin-Riou) and shows that 
\be\label{eqn:rubinsformual}
\langle\frak{S},\frak{S}\rangle=-\left[\exp^*\left(\frak{dS}_\infty\right),\log_{A,\frak{Q}}(\frak{S})\right]_{D_\frak{P}(A)}
\ee
 For $c \in H^1_f(K_{+,\frak{Q}},T_\PPP(A))$, let us define $\log_{\omega}(c) \in L$ so that 
 \be\label{eqn:logomegadefine} 
 \log_{A,\frak{Q}}(c)=\log_{\omega}(c)\cdot \overline{\omega}_{\frak{Q}}\,.
 \ee
 Combining  (\ref{en:leadingtermformulaforPRlog}), (\ref{eqn:rubinsformual}), Corollary~\ref{cor:explicitreciprocityconj} and the defining property of the Coleman-Rubin-Stark elements in Theorem~\ref{thm:mainRS}, we conclude that
\be\label{eqn:truerubinstyleformula}
-\log_{\omega}(\frak{S})\frac{\left(\mathscr{L}_\textup{cyc}^{\Sigma}\right)^\prime(\mathds{1})}{\Omega_p(\epsilon)}=(1-1/\alpha)(1-\alpha/p)^{-1}\langle\frak{S},\frak{S}\rangle\,.
\ee
Here, ${\left(\mathscr{L}_\textup{cyc}^{\Sigma}\right)^\prime(\mathds{1})}:=\displaystyle{\lim_{s\ra1}\chi_{\textup{cyc}}^{s-1}(\mathscr{L}_\textup{cyc}^\Sigma)/(s-1)}$ is the derivative of the cyclotomic restriction of the Katz' $p$-adic $L$-function, along the cyclotomic character. By our assumption that the $p$-adic height pairing is non-vanishing, it follows that $\mathscr{L}_\textup{cyc}^{\Sigma} \not\in J^2$, where $J\in \mathscr{O}[[\Gamma]]$ is the augmentation ideal. 

The Kolyvagin system method (applied with the Kolyvagin system $\bbkappa^{\textup{CRS}}$) implies in our situation that $H^1_f(K_+,T_{\frak{P}}(A))=H^1_{\FF_{\textup{Gr}}}(K_+,T_\PPP(A))$ has rank one. By the parity result of Nekov\'a\v{r} \cite[Theorem 12.2.8 (3)]{nek}, it follows that the sign of the functional equation for $L(s,\phi)$ equals $-1$. Using the generic non-vanishing results of Friedberg and Hoffstein~\cite{friedberghoffstein} and the $p$-adic Gross-Zagier formula of Disegni~\cite{disegniGZ} for a suitably chosen totally imaginary extension $F/K_+$ along with the fact that $\mathscr{L}_\textup{cyc}^{\Sigma} \not\in J^2$, we conclude that there exists a non-trivial Heegner point $P \in A(F)$. The main results of \cite{YYZ} implies that $r_{\textup{an}}(\phi_F)=1$, which in turn shows that $r_{\textup{an}}(\phi)=1$ as well.   
\end{proof}
\subsection{Logarithms of Heegner points and Perrin-Riou-Stark elements}
\label{sec:CMabvarcolemanRS}
We continue with our discussion on CM abelian varieties and our aim in this subsection is to compare the Bloch-Kato logarithm of the Coleman-Rubin-Stark element $\crs$ to the square of the logarithm of a global point on our abelian variety up to a non-zero algebraic factor.  This is a generalized form of Perrin-Riou's predictions for elliptic curves (and Beilinson-Kato elements) that we revisited in Section~\ref{sec:heegner} below in the context of elliptic curves defined over $\QQ$.

Throughout this section, we assume that $r_{\textup{an}}(\phi)=1$. We also keep working with our assumptions and conventions we have set at the start of Sections~\ref{subsubsec:CMtypesabvar}, \ref{subsubsec:CRSelement}, \ref{subsubsec:PRconj} and throughout  Section~\ref{subsubsec:grossenkaraktere}. 

We fix a quadratic and purely imaginary extension $E$ of $K_+$ such that the relative discriminant $\Delta_{E/K_{+}}$ is totally odd and relatively prime to $D_{K_+}Np$ (where $D_{K_+}$ is the discriminant of $K_+/\QQ$). We let $\eta=\eta_{E/K_+}$ the quadratic character associated to $E/K_+$. Let $N$ denote the level of the normalised new Hilbert eigenform $\phi$ of weight $2$ and suppose that $\eta(N)=(-1)^{g-1}$. Let $\phi_\eta$ denote the twisted weight 2 form, and suppose that $r_{\textup{an}}(\phi_\eta)=0$. We note that there are infinitely many choices for the field $E$ simultaneously verifying all these conditions (thanks to \cite{friedberghoffstein}); we pick one. The work of Yuan-Zhang-Zhang \cite{YYZ} applies in this situation and equips us with a \emph{Heegner point} $P_\phi \in A(K_+)$. 

Also in this case, the work of Manin, Dabrowski, Dimitrov and Januscewski associates the $p$-ordinary stabilisation of $\phi$ a $p$-adic $L$-function $L_p(\phi,\cdot) \in \LL_\frak{o}$ which is characterized by the following interpolation property (c.f., Theorem 4.4.1 of \cite{disegnithesispaper}): For every non-trivial character of $\Gamma$ of finite order\footnote{Since we assumed that $K_+/\QQ$ is unramified, it follows that any non-trivial character is ramified at all all primes of $K_+$ above $p$.} and conductor $\frak{f}_\chi$
\be\label{eqn:interpolationdmitrov}
L_p(\phi,\chi)=\chi(D_{K_+})\,{\tau(\bar{\chi})\,\mathbf{N}(\frak{f}_\chi)^{1/2}}\,{\alpha_{\frak{f_\chi}}^{-1}}\cdot\frac{L(\phi,\bar{\chi},1)}{\Omega_\phi^{+}}
\ee
where 
\begin{itemize}
\item $\tau(\bar{\chi})$ is a Gauss sum that is normalized by Disegni in loc. cit.\,; 
\item $\alpha_{\frak{f_\chi}}:=\prod_{\frak{q}\mid p} a_{\frak{q}}(\phi)^{v_{\frak{q}}(\frak{f}_\chi)}$ and $a_{\frak{q}}(\phi)$ the $p$-unit root of the Hecke polynomial at $\frak{q}$\,; 
\item $\Omega_\phi^+$ is the real period defined by Shimura and Yoshida~\cite{shimurayoshida}.
\end{itemize} 
There is likewise a $p$-adic $L$-function $L_p(\phi_\eta,\cdot)$ associated to a $p$-ordinary stabilisation of the twisted form $\phi_\eta$ (and a corresponding real period $\Omega_{\phi_\eta}^+$) as well as a $p$-adic $L$-function $L_p(\phi_E,\cdot) \in \LL_\frak{o}$ attached (by Panchishkin, Hida and Disegni) to the base change $\phi_E$. The Artin formalism yields a factorization
\be
\label{eqn:factorizationofpadicL}
L_p(\phi_E,\chi\circ {N}_{E/K_{+}})=\chi(\Delta_{E/K_{+}})^2\frac{\Omega_\phi^+\Omega_{\phi_\eta}^+}{D_E^{-1/2}\Omega_\phi}L_p(\phi,\chi)L_p(\phi_\eta,\chi)
\ee
for every character $\chi$ as above, where $\Omega_\phi=(8\pi)^2\langle \phi,\phi\rangle_{N}$ is the Shimura period. We remark that the ratio ${\Omega_\phi^+\Omega_{\phi_\eta}^+}/{\Omega_\phi}$ is always an algebraic number (that in fact belongs to the Hecke field). Set
\begin{align*}
C(\phi,E,\epsilon)&:=\mathscr{E}_{\frak{f}^+}(\psi_\epsilon)\prod_{\frak{q}|p}\left(1-{1}/{a_\frak{q}(\phi)}\right)^2\, D_F^{-1}D_E^{1/2}\frac{\Omega_{\phi}}{\Omega_\infty(\epsilon)}\cdot L(\phi_\eta,1)^{-1}\\
&=\mathscr{E}_{\frak{f}^+}(\psi_\epsilon)\prod_{\frak{q}|p}\left(1-{1}/{a_\frak{q}(\phi)}\right)^2\, D_F^{-1}D_E^{1/2}\frac{\Omega_\phi}{\Omega_\phi^+\Omega_{\phi_\eta}^+}\,\frac{\Omega_\phi^+}{\Omega_\infty(\epsilon)}\, \Omega_{\phi_\eta}^+L(\phi_\eta,1)^{-1} \in \overline{\QQ}^\times\,.
\end{align*}
where $\mathscr{E}_p(\psi_\epsilon)$ is as above and is a product of certain modified root numbers at the primes above $p$ $($and they are given as in \cite[Section 2.3]{burungaledisegni}$)$.
\begin{thm} 
\label{thm:maincolemanRSheegnerPR}
${\log_{\omega}(\frak{C})}=(1-1/\alpha)^{-1}(1-\alpha/p)\cdot C(\phi,E,\epsilon)\cdot{\log_{\omega}(P_\phi)^2}\,.$
\end{thm}
\begin{proof}
We start observing that 
\begin{align*}
\frac{\langle P_\phi,P_\phi\rangle}{\log_{\omega}(P_\phi)^2}&=\frac{\langle \frak{C},\frak{C}\rangle}{\log_{\omega}(\frak{C})^2}=-\frac{\left(\mathscr{L}_\textup{cyc}^{\Sigma}\right)^\prime(\mathds{1})}{\Omega_p(\epsilon)\cdot \log_{\omega}(\frak{C})}\cdot(1-1/\alpha)^{-1}(1-\alpha/p)\\
&=\mathscr{E}_{\frak{f}^+}(\psi_\epsilon)\,(1-1/\alpha)^{-1}(1-\alpha/p)\cdot\frac{\Omega_\phi^+}{\Omega_{\infty}(\epsilon)}\cdot\frac{L_p^\prime(\phi,\mathds{1})}{\log_{\omega}(\frak{C})}
\end{align*}
where the first equality on the first line is because  $H^1_f(K_+,T_\frak{P}(A))$ has rank one by the proof of Theorem~\ref{thm:PRconjforRSforreal} and $\langle\cdot,\cdot\rangle/\log_{\omega}(\cdot)^2$
is a non-trivial quadratic form on this space (thanks to our assumption that the $p$-adic height pairing is non-zero); the second equality on the first line is (\ref{eqn:truerubinstyleformula}) and finally, the equality on the second line follows from the Claim below. Indeed, the factor $\mathscr{E}_{\frak{f}^+}(\chi_\epsilon\psi_\epsilon)\,\chi(D_{K_+})$ that appear in the statement of Claim varies analytically in $\chi$ and tend to $\mathscr{E}_{\frak{f}^+}(\psi_\epsilon)$ as $\chi$ tends to the trivial character $\mathds{1}$.

We therefore conclude that,
\be\label{GZZhang1}
\frac{L_p^\prime(\phi,\mathds{1})}{\langle P_\phi,P_\phi\rangle}=\mathscr{E}_{\frak{f}^+}(\psi_\epsilon)^{-1}\,(1-1/\alpha)(1-\alpha/p)^{-1}\,\frac{\Omega_{\infty}(\epsilon)}{\Omega_\phi^+}\,\frac{\log_{\omega}(\frak{C})}{\log_{\omega}(P_\phi)^2}\,.
\ee
On the other hand,
\begin{align}
\notag\frac{L_p^\prime(\phi,\mathds{1})}{\langle P_\phi,P_\phi\rangle}&=\frac{L_p^\prime(\phi_E,\mathds{1})}{\langle P_\phi,P_\phi\rangle}\cdot\frac{D_E^{1/2}\Omega_{\phi}}{\Omega_{\phi}^+}\cdot (\Omega_{\phi_\eta}^+\cdot L_p(\phi_\eta,\mathds{1}))^{-1}\\
\notag&=\frac{L_p^\prime(\phi_E,\mathds{1})}{\langle P_\phi,P_\phi\rangle}\cdot\frac{D_E^{1/2}\Omega_{\phi}}{\Omega_{\phi}^+}\cdot\left({\prod_{\frak{q}|p}\left(1-\frac{1}{\eta(\frak{q})a_\frak{q}(\phi)}\right)^2}L(\phi_\eta,1)\right)^{-1}\\
\label{GZZhang2}&=\prod_{\frak{q}|p}\left(1-{1}/{a_\frak{q}(\phi)}\right)^2D_F^{-1}D_E^{1/2}\,\frac{\Omega_{\phi}}{\Omega_{\phi}^+}\, L(\phi_\eta,1)^{-1}
\end{align}
where the first equality follows from the factorization (\ref{eqn:factorizationofpadicL}) of the base-change $p$-adic $L$-function, the second from the interpolation formula for the twisted $p$-adic $L$-function and the last from the $p$-adic Gross-Zagier formula of Disegni~\cite[Theorem B]{disegniGZ}. Using (\ref{GZZhang1})  together with (\ref{GZZhang2}), we conclude that 
$$\frac{\log_{\omega}(\frak{C})}{\log_{\omega}(P_\phi)^2}=(1-1/\alpha)^{-1}(1-\alpha/p)\,\mathscr{E}_{\frak{f}^+}(\psi_\epsilon)\prod_{\frak{q}|p}\left(1-{1}/{a_\frak{q}(\phi)}\right)^2\, D_F^{-1}D_E^{1/2}\frac{\Omega_{\phi}}{\Omega_\infty(\epsilon)}\cdot L(\phi_\eta,1)^{-1}$$
and the proof follows.
\end{proof}
\begin{claim}
For all sufficiently ramified characters $\chi$ of $\Gamma$ the following identity holds:
$$\frac{\mathscr{L}_\textup{cyc}^{\Sigma}(\chi)}{\Omega_p(\epsilon)}=\mathscr{E}_{\frak{f}^+}(\chi_\epsilon\psi_\epsilon)\,\chi(D_{K_+})\,\Omega_{\phi}^+\,L_p(\phi,\chi^{-1})\,.$$
\end{claim}
\begin{proof}
This will follow once we match the interpolation factors in Corollary~\ref{cor:padicLfunctionforA} for $\mathscr{L}_\textup{cyc}^{\Sigma}(\chi)$ and (\ref{eqn:interpolationdmitrov}) for $L_p(\phi,\chi^{-1})$. More precisely, we would like to verify that $\mathscr{E}_{p}(\chi_\epsilon\psi_\epsilon)$ equals ${\tau({\chi})\,\mathbf{N}(\frak{f}_\chi)^{1/2}}\,{\alpha_{\frak{f_\chi}}^{-1}}$ for all sufficiently ramified characters $
\chi$. The proof of this claim is essentially contained in \cite[Section 2.3]{burungaledisegni} (although we also rely on Hsieh's exposition in \cite[Section 4.7]{hsiehanticyclopadicL}), we provide an outline here for the sake of completeness. For all such $\chi$, the local $L$-factors are trivial and by Tate's local functional equation, it follows that 
$$\mathscr{E}_{p}(\chi_\epsilon\psi_\epsilon)=\prod_{\wp\in \Sigma_p^c}\tau(\psi_{\epsilon,\wp}\chi_{\epsilon,\wp}\,, \Psi_\wp)\,,$$ 
where $\tau(\psi_{\epsilon,\wp}\chi_{\epsilon,\wp}, \Psi_\wp)$ are the Gauss sums which are (un)normalized as in \cite{disegniGZ} and $ \Psi_\wp$ are a suitably determined local additive characters. Let $\alpha_\frak{P}$ be the local character of $K_{+,\frak{P}}^\times$ such that the local constituent of $\phi$ at $\frak{P}=\wp\wp^c$ is the irreducible principal series $\pi(\alpha_\frak{P},\beta_\frak{P})$ for some other local character $\beta_\frak{P}$. Then $\tau(\psi_{\epsilon,\wp}\chi_{\epsilon,\wp}\,, \Psi_\wp)=\tau(\alpha_{\frak{P}}\chi_{\epsilon,\frak{P}}, \Psi_\frak{P})$, where we write $\chi_{\epsilon,\frak{P}}$ for the local character $\chi_{\epsilon,\wp}$ of $K_{+,\frak{P}}^\times=K_{\wp}^\times$ and we similarly define the additive character $\Psi_\frak{P}$. We therefore infer that 
$$\mathscr{E}_{p}(\chi_\epsilon\psi_\epsilon)=\prod_{\frak{P}\mid p}\tau(\alpha_{\frak{P}}\chi_{\epsilon,\frak{P}}, \Psi_\frak{P})\,.$$
The  $\tau(\chi)$ of \cite{disegnithesispaper} is precisely the normalization of $\prod_{\frak{P}\mid p}\tau(\chi_{\frak{P}},\Psi_{\frak{P}})$, so that 
$$\tau(\chi)\,\mathbf{N}(\frak{f}_\chi)^{1/2}=\prod_{\frak{P}\mid p}\tau(\chi_{\epsilon,\frak{P}},\Psi_{\frak{P}})\,.$$ Finally, as explained in the proof of Lemma A.1.1 of \cite{disegniGZ} we have 
$$\tau(\chi_{\epsilon,\frak{P}},\Psi_{\frak{P}})= a_{\frak{q}}(\phi)^{v_{\frak{q}}(\frak{f}_\chi)}\,\tau(\alpha_{\frak{P}}\chi_{\epsilon,\frak{P}}, \Psi_\frak{P})$$
and hence,
\begin{align*}{{\alpha_{\frak{f_\chi}}^{-1}}\,\tau({\chi})\,\mathbf{N}(\frak{f}_\chi)^{1/2}}&=\alpha_{\frak{f_\chi}}^{-1}\,\prod_{\frak{P}\mid p}\tau(\chi_{\epsilon,\frak{P}},\Psi_{\frak{P}})\\
&=\prod_{\frak{P}\mid p}\tau(\alpha_{\frak{P}}\chi_{\epsilon,\frak{P}}, \Psi_\frak{P})=\mathscr{E}_{p}(\chi_\epsilon\psi_\epsilon)\,.
\end{align*}
\end{proof}
\begin{rem}
\label{rem:chiportions}
Assuming that the sign of the functional equation for the Hecke character $\psi_\epsilon^*:=\psi_{\epsilon}\circ c$ equals $+1$, Burungale and Disegni proved in \cite{burungaledisegni} that the $p$-adic height pairing 
$$\langle\,,\,\rangle_\chi: H^1_f(K,T_\wp(A)\otimes\chi)\otimes H^1_f(K,T_{\wp^c}(A)\otimes\chi^{-1})\lra M$$
for almost all anticyclotomic Hecke characters $\chi$ of $K$ of finite order (where $M$ is a finite extension of $\QQ_p$ in which $\chi$ takes its values). 
\end{rem}

\subsection*{Acknowledgements}
I would to express my gratitude to Masataka Chida for carefully reading an earlier version of this manuscript and for his comments which helped me improve parts of the text, and to Daniel Disegni and Karl Rubin for very helpful correspondence. Special thanks are due to Henri Darmon for generously sharing his insights with me, for his suggestions and his constant encouragement. I also would like to thank Denis Benois, whose question on the Tate-Shafarevich groups of higher weight modular forms started the train of thoughts which lead me to the main results of this article.
{\bibliographystyle{halpha}
\bibliography{references}}
\end{document}